\documentclass[reqno]{amsart}
\usepackage{amssymb,amsmath,amsthm,mathtools,mathrsfs,xcolor}
\usepackage[english]{babel}
\usepackage{cite}
\RequirePackage[colorlinks,citecolor=blue,urlcolor=blue]{hyperref}

\DeclareMathOperator{\Var}{Var}

\DeclareMathOperator{\diam}{diam}

\newcommand{\R}{\mathbb R}
\newcommand{\N}{\mathbb N}
\newcommand{\E}{\mathbb E}
\renewcommand{\P}{\mathbb P}
\renewcommand{\d}{\mathrm d}
\newcommand{\e}{\mathrm e}
\newcommand{\1}{\mathbf 1}

\newcommand{\inner}[2]{\langle #1,#2\rangle}

\numberwithin{equation}{section}
\newtheorem{theorem}{Theorem}[section]
\newtheorem{proposition}[theorem]{Proposition}
\newtheorem{lemma}[theorem]{Lemma}
\newtheorem{corollary}[theorem]{Corollary}
\theoremstyle{definition}

\title[Stochastic heat equation]{Sharp regularity and small ball probabilities for the stochastic heat equation on bounded domains}

\author[J. Hu and C. Y. Lee]{Jingwu Hu \and Cheuk Yin Lee}
\address{Department of Statistics and Probability, Michigan State University, East Lansing, MI 48824, United States}
\email{hujingwu@msu.edu}
\address{School of Science and Engineering, The Chinese University of Hong Kong (Shenzhen), Longgang, Shenzhen, Guangdong 518172, China}
\email{leecheukyin@cuhk.edu.cn}


\begin{document}

\begin{abstract}
We consider the stochastic heat equation
\[
	\partial_t u(t,x) = \Delta u(t,x) + \dot{W}_\alpha(t,x)
\]
on a bounded Lipschitz domain with zero Dirichlet boundary condition and zero initial condition, where $\dot{W}_\alpha$ is a Gaussian noise that is white in time and whose spatial covariance is the kernel of $(-\Delta)^{-\alpha}$ with $\alpha>0$. We prove that a unique pointwise defined mild solution exists if and only if $\alpha>d/2-1$. In this case, if in addition the domain is $C^2$, we also establish spatial and temporal H\"older regularity of the solution.
When $d/2-1<\alpha<d/2$, we show that the H\"older exponents are optimal and obtain exact local and uniform moduli of continuity, a Chung-type law of the iterated logarithm, and sharp small ball probability estimates for the solution.
\end{abstract}

\maketitle

\section{Introduction}

Let \(D\subset\R^d\) be a bounded domain. 
As usual, a domain is understood to be a nonempty open connected set. 
We fix a parameter $\alpha > 0$ and consider the stochastic heat equation with Dirichlet boundary condition:
\begin{equation}\label{eq:SHE}
	\begin{cases}
		\partial_t u(t,x)=\Delta u(t,x)+\dot W_\alpha(t,x),
    	& t>0,\ x\in D,\\
		u(t,x)=0, & t>0,\ x\in\partial D,\\
		u(0,x)=0, & x\in D.
	\end{cases}
\end{equation}
Formally, the noise $\dot W_\alpha$ is given by $\dot W_\alpha = (-\Delta)^{-\alpha/2} \dot W$, where $\dot W$ is a space-time white noise on $[0,\infty) \times D$ and $-\Delta$ denotes the positive Dirichlet Laplacian on $D$.
More precisely, $\dot W_\alpha$ is defined as a centered generalized Gaussian field on a complete probability space $(\Omega, \mathcal{F}, \P)$, indexed by $(t,x) \in [0,\infty) \times D$, with covariance
\begin{align}\label{W:corr:G}
    \E[\dot W_\alpha(t,x) \dot W_\alpha(s,y)] = \delta_0(t-s) G_{\alpha}(x,y),
\end{align}
where $G_{\alpha}$ is the kernel of the operator $(-\Delta)^{-\alpha}$.
As in standard theory of stochastic partial differential equations (SPDEs), the solution to \eqref{eq:SHE} is interpreted as mild solution \cite{Walsh,Dalang,DPZ}; see \eqref{eq:mild} below.



If $\dot{W}_\alpha$ is space-time white noise (i.e., $\alpha=0$), then
\eqref{eq:SHE} has a pointwise solution only when $d=1$.
In this case, sample path properties of the linear and nonlinear stochastic heat equation on a bounded interval have been studied in \cite{HL26}.
In this paper, we will focus on the linear case for $\alpha>0$.
%
%
%
%
%
When $D=\R^d$, it is common to consider spatially-homogeneous Gaussian noise which is white in space or has spatial covariance given by the Riesz kernel $|x-y|^{-\beta}$, where $0<\beta<d$ (see \cite{Dalang}).
Although the Riesz kernel may also be used on bounded domains \cite{CCL24}, it is often more natural to use spatial covariances that are similar to the one in \eqref{W:corr:G} on general bounded domains, such as $d$-dimensional torus \cite{CCV}, Riemannian manifolds \cite{CO25}, and metric measure spaces \cite{BCHOTW}.
In fact, when $0<\alpha<d/2$, the covariance $G_\alpha(x,y)$ behaves like the Riesz kernel with $\beta = d-2\alpha$ (see Lemma \ref{lem:G:bd} below).
Since the noise $\dot{W}_\alpha(t,x)$ is by definition a generalized Gaussian field and may not be a pointwise function, an important question is to determine when there exists a pointwise solution and determine regularity of the solution.

There is a large literature on H\"older regularity of linear and nonlinear stochastic heat equations and related SPDEs on $\R_+ \times \R^d$; see, e.g., \cite{Walsh,CD14,SS00,SS02,BC18,HL19,BJQ16,BQS19,KS23,C17,L17,DS26}.
In particular, \cite{KS23} studied optimal H\"older regularity of SPDEs on $\R_+ \times \R^d$ with additive Gaussian noise, including the stochastic heat equation.
H\"older and Sobolev regularity of stochastic heat equations and parabolic SPDEs on bounded domains have also been studied in \cite{CCV,DKZ,SV,NV09,DPZ,DS26}, but there are not many explicit results about optimality of H\"older regularity on general domains.

The goal of this paper is to study sharp H\"older regularity of the solution to \eqref{eq:SHE} on bounded domains and establish further sample path properties including exact moduli of continuity, Chung's law of the iterated logarithm, and sharp small ball probability estimates.

Sample path properties of Gaussian random fields and SPDEs have been studied extensively. Exact local and uniform moduli of continuity and Chung's laws of the iterated logarithm for Gaussian processes and Gaussian random fields have been established in \cite{MWX13,LX23,LZ26,LX10}, while similar sample path results for stochastic heat equations and other related SPDEs can be found in \cite{FKM15,HSWX20,Chen23,GSWX25,WX24,HL26,LX19,CLX26}.
In many cases, the local modulus functions are given by the law of the iterated logarithm (LIL) and the uniform modulus functions are similar to L\'evy's modulus of continuity for Brownian motion. 
These moduli of continuity specify the upper envelope ($\limsup$) for the increments at small scales.
The lower envelope ($\liminf$) for the increments is usually characterized by Chung's LIL, originally due to Chung \cite{Chung} for Brownian motion.
It is a well-known fact that Chung's LIL is closely related to small ball probabilities \cite{LS01}.
We refer to \cite{AJM21,C24,C25,C26,KKM24,H24,FJK23,HL26} for small ball probability results for stochastic heat equations and related SPDEs.

\subsection{Main results}

We first present a necessary and sufficient condition for the existence of a unique random-field solution to \eqref{eq:SHE}. 
Such condition is also known as Dalang's condition \cite{Dalang}.
According to standard SPDE theory \cite{Walsh, Dalang, DPZ}, \eqref{eq:SHE} has a unique pointwise mild solution given by
\begin{equation}\label{eq:mild}
    u(t,x)=\int_0^t\int_D P_{t-r}(x,y)\,W_\alpha(\d r,\d y)
\end{equation}
if and only if the above Wiener integral process is well defined pointwise, where $P$ denotes the Dirichlet heat kernel.

\begin{theorem}[Dalang condition]
\label{thm:exist}
Suppose $D \subset \R^d$ is a bounded Lipschitz domain.
Then \eqref{eq:mild} is pointwise defined for every $(t,x) \in (0,\infty)\times D$ if and only if
\begin{align}\label{Dalang}
	\alpha>d/2-1.
\end{align}
\end{theorem}

We obtain H\"older regularity for the solution when in addition $D$ is $C^2$.

\begin{theorem}[H\"older regularity]\label{thm:reg}
Suppose $D \subset \R^d$ is a bounded $C^2$ domain and \eqref{Dalang} holds.
Then $\{u(t,x)\}_{t \ge 0, x \in D}$ has a version that is locally H\"older continuous of order $\beta_0$ in $t$ and order $\beta_1$ in $x$ for any $\beta_0 \in (0, \frac{\theta\wedge 1}{2})$
and $\beta_1 \in (0,\theta\wedge 1)$, where
\begin{align}\label{theta}
	\theta=\alpha-d/2+1.
\end{align}
\end{theorem}

Next, we establish the optimality of the H\"older exponents and sharp moduli of continuity under the condition
\begin{align}\label{alpha:Holder:range}
	d/2-1<\alpha<d/2,
\end{align}
which is equivalent to $0<\theta<1$.
In this case, we define the metric $\rho$ on $[0,\infty) \times D$ by
\begin{equation}\label{eq:rho}
    \rho((t,x),(s,y))
    =|t-s|^{\theta/2}+\sum_{i=1}^d |x_i-y_i|^{\theta},
\end{equation}
for any $(t,x) = (t,x_1,\dots, x_d), (s,y) = (s,y_1,\dots, y_d) \in [0,\infty)\times D$.
Note that $\rho$ is equivalent to the metric $|t-s|^{\theta/2} + |x-y|^\theta$.
The next theorem gives sharp upper and lower bounds for the variance of increments including exact dependence near the boundary and at small $t$, and shows that the solution satisfies the strong local nondeterminism property under condition \eqref{alpha:Holder:range}.
This implies that the H\"older exponents in Theorem \ref{thm:reg} are optimal under this condition.
We say that a centered Gaussian random field $\{ X(s) \}_{s \in I \subset \R^N}$ satisfies strong local nondeterminism (SLND) if there exists $C>0$ such that
\begin{align}\label{SLND}
	\Var(X(s)\mid X(s_1),\dots, X(s_n)) \ge C \min_{1\le i\le n} \Var(X(s)-X(s_i)),
\end{align}
uniformly for all $n \in \N_+$ and $s,s_1,\dots, s_n \in I$ \cite{Berman, Pitt, CD82, MP87, Xiao08}.
SLND has been a useful tool for studying various sample properties of Gaussian processes and Gaussian random fields; see \cite{X09,Xiao08} and the references therein.
One of the main contributions of this paper is showing the SLND property for the solution.

\begin{theorem}\label{thm:optimal}
Suppose $D \subset \R^d$ is a bounded $C^2$ domain and \eqref{alpha:Holder:range} holds.
Then for any $T>0$,
\begin{align}\label{var:u-u}
	\Var(u(t,x)-u(s,y)) \asymp \left[\rho((t,x),(s,y)) \wedge \left((\sqrt{t} \wedge d(x))^\theta \vee (\sqrt{s} \wedge d(y))^\theta\right)\right]^2
\end{align}
uniformly for all $t,s \in [0,T]$ and $x,y \in D$, where $\theta$ is given by \eqref{theta}, $\rho$ is given by \eqref{eq:rho}, and $d(x)$ is the distance-to-boundary defined by
\begin{align}\label{d(x)}
	d(x):= \mathrm{dist}(x,\partial D).
\end{align}
In particular, setting $s=0,y=x$ yields $\Var(u(t,x)) \asymp t^{\theta} \wedge (d(x))^{2\theta}$ on $[0,T]\times D$.
Moreover, for any $T>\delta>0$ and compact subset $D' \subset D$, $\{u(t,x)\}_{(t,x) \in [\delta,T]\times D'}$ satisfies SLND in the sense of \eqref{SLND} with $C$ depending on $\delta,T,D'$.
\end{theorem}



As a consequence of the optimal regularity and the SLND property, we obtain sharp local and uniform moduli of continuity, Chung's law of the iterated logarithm, and sharp small ball probability estimates for the solution.
To present these results, let us define
\begin{align}
	&B_\rho((t,x),r) := \{ (s,y) \in [0,\infty) \times D : \rho((t,x),(s,y)) \le r \},\\
	&B^*_\rho((t,x),r) := \{ (s,y) \in [0,\infty) \times D : 0< \rho((t,x),(s,y)) \le r\}
\end{align}
for any $(t,x) \in [0,\infty) \times D$ and $r>0$.

\begin{theorem}[Exact local modulus of continuity]\label{thm:moc}
Suppose \(D\) is a bounded \(C^2\) domain and
\eqref{alpha:Holder:range} holds. 
Then, for every
\(z_0=(t_0,x_0)\in(0,\infty)\times D\), there exists a nonrandom number 
\(K_0 \in(0,\infty)\) depending on $z_0$ such that
\[
    \lim_{\varepsilon\to0^+}
    \sup_{z\in B_\rho^*(z_0,\varepsilon)}
    \frac{|u(z)-u(z_0)|}
    {\rho(z,z_0)\sqrt{\log\log(1/\rho(z,z_0))}}
    =
    K_0
    \qquad\text{a.s.}
\]
\end{theorem}

\begin{theorem}[Exact uniform modulus of continuity]\label{thm:umoc}
Suppose \(D\) is a bounded \(C^2\) domain and
\eqref{alpha:Holder:range} holds. 
Then, for any $T>0$, there exists a nonrandom number \(K_1 \in(0,\infty)\) depending on $T$ such that
\begin{align}\label{umoc}
    \lim_{\varepsilon\to0^+}
    \sup_{\substack{z,z'\in [0,T]\times D\\0<\rho(z,z')\le\varepsilon}}
    \frac{|u(z)-u(z')|}
    {\rho(z,z')\sqrt{\log(1/\rho(z,z'))}}
    =
    K_1
    \qquad\text{a.s.}
\end{align}
\end{theorem}

\begin{theorem}[Chung's law of the iterated logarithm]\label{thm:lil}
Suppose \(D\) is a bounded \(C^2\) domain and
\eqref{alpha:Holder:range} holds. 
Then, for every \(z_0=(t_0,x_0)\in(0,\infty)\times D\), there exists a nonrandom number \(K_2 \in(0,\infty)\) depending on $z_0$ such that
\[
    \liminf_{\varepsilon\to0^+}
    \frac{
        \sup_{z\in B_\rho(z_0,\varepsilon)}
        |u(z)-u(z_0)|
    }{
        \varepsilon(\log\log(1/\varepsilon))^{-1/Q}
    }
    =
    K_2
    \qquad\text{a.s.,}
\]
where
\begin{equation}\label{eq:Q}
    Q=\frac{2}{\theta}+ \sum_{i=1}^d\frac{1}{\theta}
    =
    \frac{2+d}{\alpha-d/2+1}.
\end{equation}
\end{theorem}

\begin{theorem}[small ball probability]\label{thm:sbp}
Suppose \(D\) is a bounded \(C^2\) domain and \eqref{alpha:Holder:range} holds. Then, for any $T>0$, there exist \(K_3,K_4>0\) depending on $T$ such that
\[
    \exp\left(-K_3\varepsilon^{-Q}\right)
    \le
    \P\left\{
        \sup_{(t,x)\in [0,T]\times D}|u(t,x)|\le\varepsilon
    \right\}
    \le
    \exp\left(-K_4\varepsilon^{-Q}\right),
\]
uniformly for all \(\varepsilon \in (0, 1]\), where $Q$ is given by \eqref{eq:Q}.
\end{theorem}

Matching small ball probability bounds for nonlinear stochastic heat equations driven by space-time white noise on a bounded interval have been established in \cite{AJM21}.
The case of nonlinear stochastic heat equations on torus and compact Riemannian manifolds driven by temporally-white, spatially-colored Gaussian noise have also been studied in \cite{C24,C25,C26}.
However, in \cite{C24,C25,C26}, the exponents in the upper and lower bounds do not match.
In Theorem \ref{thm:sbp} above, we are able to obtain matching bounds for the linear stochastic heat equation \eqref{eq:SHE} driven by colored noise of a similar type on bounded $C^2$ domains, thanks to our SLND property.

Theorems \ref{thm:moc}--\ref{thm:sbp} are established using the framework of \cite{LX23}, which relies on sharp variance bounds for the increments and the SLND property.
Sharp variance bounds in the case of bounded domains are more involved compared to the case of bounded intervals in \cite{HL26}.
In particular, the Dirichlet eigenfunctions are not uniformly bounded unless $d=1$ and uniform bounds of the Dirichlet eigenvalues and eigenfunctions are not good enough to derive those sharp variance bounds in higher dimensions.
It turns out that sharp variance estimates require additional tools such as Gaussian-type heat kernel estimates.
Also, the domain is required to be sufficiently smooth (with $C^2$ boundary) in order to apply gradient estimates for the heat kernel to obtain sharp spatial regularity.
Similarly to \cite{HL26}, the SLND property is established using the orthonormal basis of eigenfunctions and a scaling argument involving suitable test functions.

Since the temporal process $\{u(t,x_0)\}_{t \in [\delta,T]}$ and the spatial process $\{ u(t_0,x) \}_{x \in D'}$ also satisfy SLND, the following results can be obtained using the same proofs for the above main theorems.

\begin{corollary}
Suppose $D\subset \R^d$ is a bounded $C^2$ domain and \eqref{alpha:Holder:range} holds.
Then, for any fixed $T \ge t_0>0$ and $x_0 \in D$, there exist constants $C_0,C_0',C_1,C_1',C_2,C_2' \in (0,\infty)$ such that
\begin{gather*}
	\lim_{\varepsilon\to0^+} \sup_{t \ge 0:0< |t-t_0| \le \varepsilon} \frac{|u(t,x_0)-u(t_0,x_0)|}{|t-t_0|^{\theta/2} \sqrt{\log\log(1/|t-t_0|)}} = C_0 \quad \text{a.s.},\\
	\lim_{\varepsilon\to0^+} \sup_{x \in D: 0< |x-x_0| \le \varepsilon} \frac{|u(t_0,x)-u(t_0,x_0)|}{|x-x_0|^\theta \sqrt{\log\log(1/|x-x_0|)}} = C_0' \quad \text{a.s.},\\
	\lim_{\varepsilon\to0^+} \sup_{t,t' \in [0,T]: 0<|t-t'| \le \varepsilon} \frac{|u(t,x_0)-u(t',x_0)|}{|t-t'|^{\theta/2} \sqrt{\log(1/|t-t'|)}} = C_1 \quad \text{a.s.},\\
	\lim_{\varepsilon\to0^+} \sup_{x,x'\in D: 0<|x-x'| \le \varepsilon} \frac{|u(t_0,x)-u(t_0,x')|}{|x-x'|^\theta \sqrt{\log(1/|x-x'|)}} = C_1' \quad \text{a.s.},\\
	\liminf_{\varepsilon\to0^+} \frac{\sup_{t: |t-t_0| \le \varepsilon}|u(t,x_0)-u(t_0,x_0)|}{\varepsilon^{\theta/2} (\log\log(1/\varepsilon))^{-\theta/2}} = C_2 \quad \text{a.s.},\\
	\liminf_{\varepsilon\to0^+} \frac{\sup_{x \in D: |x-x_0| \le \varepsilon}|u(t_0,x)-u(t_0,x_0)|}{\varepsilon^\theta (\log\log(1/\varepsilon))^{-\theta/d}} = C_2' \quad \text{a.s.}
\end{gather*}
Moreover, there exist $C_3,C_3',C_4,C_4' \in (0,\infty)$ such that
\begin{gather*}
	\exp(-C_3 \varepsilon^{-2/\theta}) \le \P\left\{ \sup_{t \in [0,T]} |u(t,x_0)| \le \varepsilon \right\} \le \exp(-C_4 \varepsilon^{-2/\theta}),\\
	\exp(-C_3' \varepsilon^{-d/\theta}) \le \P\left\{ \sup_{x \in D} |u(t_0,x)| \le \varepsilon \right\} \le \exp(-C_4' \varepsilon^{-d/\theta}),
\end{gather*}
uniformly for all $\varepsilon \in (0,1]$.
\end{corollary}


The rest of the paper is organized as follows. 
In Section \ref{s:pre}, we recall some basic facts about the Dirichlet heat kernel and fractional Laplacian, and give the precise definition of the Gaussian noise and corresponding stochastic integrals.
In Section \ref{s:exist}, we prove Theorem \ref{thm:exist}, derive variance bounds, and give a series representation for the solution.
In Section \ref{sec:regularity}, we prove Theorem \ref{thm:reg} and obtain temporal and spatial regularity of the solution.
In Section \ref{s:optimal}, we show that the solution satisfies SLND and prove Theorem \ref{thm:optimal}.
In Section \ref{s:proofs}, we prove Theorems \ref{thm:moc}--\ref{thm:sbp}.

\subsection{Notation}
We clarify some notation that will be used throughout the paper:
\(\mathbb N_+=\{1,2,\ldots\}\); \(\mathbb N_0=\{0,1,2,\ldots\}\);
\(\R_+=(0,\infty)\);
For a set \(A\), \(\#A\) denotes its cardinality and \(\1_A\) its indicator
function; \(a\wedge b=\min\{a,b\}\),
\(a\vee b=\max\{a,b\}\); $\log_+(x) = \log(e \vee x)$, where $\log$ denotes natural logarithm; \(c,C\) denote finite positive constants
whose values may change from line to line; For two functions \(f\)
and \(g\), ``\(f\lesssim g\)'' means there is $C>0$ such that \(f(x)\le Cg(x)\) for all $x$; ``\(f\asymp g\)'' means
\(f\lesssim g\) and \(g\lesssim f\); ``\(f(x)=O(g(x))\)'' means \(|f|\lesssim |g|\); For any $k \ge 1$, $\|X\|_k = (\E|X|^k)^{1/k}$ denotes the $L^k(\Omega)$-norm of a random variable $X$; $B(x,r)$ and $\overline{B}(x,r)$ respectively denote open and closed Euclidean balls centered at $x$ with radius $r$; $|x|$ denotes the Euclidean norm of $x$.

\section{Preliminaries}\label{s:pre}

\subsection{Dirichlet Heat kernel}

Let $D\subset \R^d$ be a bounded Lipschitz domain.
According to standard spectral theory \cite{Davies,McLean}, the Dirichlet Laplacian $-\Delta$ has a discrete spectrum, with strictly positive eigenvalues $0<\lambda_1 < \lambda_2 \le \lambda_3 \le \dots$ and an orthonormal basis of eigenfunctions $\{f_n\}_{n=1}^\infty$ for $L^2(D)$, with respect to the inner product $\langle \phi\,,\psi\rangle_{L^2(D)} = \int_D \phi(x) \psi(x)\, \d x$, such that $-\Delta f_n = \lambda_n f_n$ in $D$ and $f_n=0$ on $\partial D$.
Moreover, $f_1$ can be chosen to be positive \cite[Proposition 1.4.3]{Davies} and $f_n \in H^1_0(D) \cap C^\infty(D)$ \cite[Theorem 9.31]{Brezis}.
The Dirichlet heat kernel is given by
\begin{align}\label{P}
    P_t(x,y) = \sum_{n=1}^\infty \e^{-\lambda_n t} f_n(x) f_n(y) \qquad \forall t > 0, x, y \in D.
\end{align}
Due to Weyl's law $\lambda_n \asymp n^{2/d}$ (see \cite[Corollary 3.15]{FLW}) and $\sup_{x \in D}|f_n(x)|\lesssim \lambda_n^{d/4}$ (see \cite[Proposition 5]{F95}), for any $t>0$ the sum in \eqref{P} is uniformly absolutely convergent for $x,y\in D$.
Also, it follows from \eqref{P} that $P_t$ is symmetric, i.e., $P_t(x,y) = P_t(y,x)$ for all $t>0$ and $x,y\in D$, and $P_t$ has the semigroup property:
\begin{align}\label{semigroup}
    \int_D P_t(x,y) P_s(y,z) \, \d y = P_{t+s}(x,z) \quad \forall t,s>0, x,z \in D.
\end{align}
Recall a two-sided heat kernel estimate \cite{Riahi}:
\begin{lemma}\label{lem:P}
    If $D \subset \R^d$ is a bounded Lipschitz domain, then there exist constants $c_1,c_2,C_1,C_2>0$ and $0<\mu \le 1 \le \nu$ such that
    \begin{align}\begin{split}\label{P:bd}
        &c_1\left( 1\wedge \frac{f_1(x)}{1 \wedge t^{\mu/2}} \right)\left( 1\wedge \frac{f_1(y)}{1 \wedge t^{\mu/2}} \right) \frac{\e^{-\lambda_1 t}}{1\wedge t^{d/2}}
        \exp\left(-\frac{|x-y|^2}{C_1t}\right)
        \le P_t(x,y) \\
        &
        \le c_2 \left( 1\wedge \frac{f_1(x)}{1 \wedge t^{\nu/2}} \right)\left( 1\wedge \frac{f_1(y)}{1 \wedge t^{\nu/2}} \right) \frac{\e^{-\lambda_1 t}}{1\wedge t^{d/2}}
        \exp\left(-\frac{|x-y|^2}{C_2t}\right)
    \end{split}\end{align}
    uniformly for all $t>0$ and $x,y \in D$. If in addition $D$ is $C^{1,\gamma}$ for some $\gamma>0$, then \eqref{P:bd} holds with $\mu=\nu=1$ and $f_1(x) \asymp d(x)$ for all $x \in D$, where $d(x)$ is the distance-to-boundary given by \eqref{d(x)}.
\end{lemma}

When in addition $D$ is a $C^2$ domain, we also have the following gradient estimate for the heat kernel:

\begin{lemma}\label{lem:grad:P}
If $D \subset \R^d$ is a bounded $C^2$ domain, then there exists $C>0$ such that
\begin{align*}
	|\nabla_x P_t(x,y)| \le \frac{e^{-\frac12 \lambda_1 t}}{t^{(d+1)/2}} \exp\left(-\frac{|x-y|^2}{Ct}\right),
\end{align*}
uniformly for all $t>0$ and $x,y \in D$.
\end{lemma}

\begin{proof}
By Theorem 2.1 of \cite{Z06}, when $D$ is a bounded $C^2$ domain, there exists $c>0$ such that
\begin{align*}
	|\nabla_x \log P_t(x,y)| \le 
	\begin{cases}
	\frac{c}{d(x)} & \text{if $d(x) \le \sqrt{t}$,}\\
	\frac{c}{\sqrt{t}} \left( 1+ \frac{|x-y|}{\sqrt{t}}\right) & \text{if $d(x)>\sqrt{t}$,}
	\end{cases}
\end{align*}
uniformly for all $t \in (0,1]$ and $x,y \in D$.
It follows that
\begin{align*}
	|\nabla_x P_t(x,y)| \le 
	\begin{cases}
	\frac{c}{d(x)}P_t(x,y) & \text{if $d(x) \le \sqrt{t}$,}\\
	\frac{c}{\sqrt{t}} \left( 1+ \frac{|x-y|}{\sqrt{t}}\right)P_t(x,y) & \text{if $d(x)>\sqrt{t}$,}
	\end{cases}
\end{align*}
uniformly for all $t \in (0,1]$ and $x \in D$.
Using Lemma \ref{lem:P} in both cases and using also the inequality $z e^{-z^2} \lesssim e^{-z^2/2}$ for $z \ge 0$ in the  $d(x)>\sqrt{t}$ case, we deduce that
\begin{align}\label{grad:log:P}
	|\nabla_x P_t(x,y)| \lesssim \frac{e^{-\lambda_1 t}}{t^{(d+1)/2}} \exp\left( - \frac{|x-y|^2}{2C_2 t} \right),
\end{align}
uniformly for all $t \in (0,1]$ and $x,y \in D$.
It remains to show the same estimate for $t\ge 1$.
In this case, we use the semigroup property \eqref{semigroup} to write
\[
	\nabla_x P_t(x,y) = \int_D \nabla_x P_{1/2}(x,z) P_{t-1/2}(z,y) \, \d z
\]
for all $t \ge 1$ and $x,y \in D$, where the exchange of $\nabla_x$ and the integral can be justified using the dominated convergence theorem and the gradient estimate for $t=1/2$ established above.
Then, by \eqref{grad:log:P} and Lemma \ref{lem:P}, we have
\begin{align*}
	|\nabla_x P_t(x,y)| &\lesssim e^{-\lambda_1(t-1/2)} \int_D  \exp\left( - \frac{|x-z|^2}{C_2} \right) \, \d z\\
	& \lesssim e^{-\lambda_1 t} \lesssim \frac{e^{-\frac12\lambda_1 t}}{t^{(d+1)/2}} \exp\left( - \frac{|x-y|^2}{2C_2t} \right),
\end{align*}
uniformly for all $t \ge 1$ and $x,y \in D$, where the last inequality holds because $e^{-\lambda_1 t} \lesssim e^{-\frac12 \lambda_1 t}\,t^{-(d+1)/2}$ for $t \ge 1$ and the exponent factor is bounded below for all $t \ge 1$ and $x,y \in D$.
\end{proof}

\subsection{Fractional Laplacian}

For any $\alpha>0$, the fractional Laplacian $(-\Delta)^{-\alpha}$ is the bounded linear operator on $L^2(D)$ defined by
\begin{align}\label{frac:Lapl}
    (-\Delta)^{-\alpha} \phi = \sum_{n=1}^\infty \lambda_n^{-\alpha}\langle \phi,f_n \rangle_{L^2(D)}\, f_n \qquad \forall \phi \in L^2(D).
\end{align}
Following \cite{BC23, BCHOTW}, we define the Schwartz space on $D$ by
\begin{align}\label{S(D)}
	\mathcal{S}(D) = \left\{ \phi \in C_0(D) : \forall k \ge 0, \lim_{n\to \infty} n^k\left|\langle \phi, f_n \rangle_{L^2(D)}\right| = 0 \right\}.
\end{align}

\begin{lemma}
Let $D \subset \R^d$ be a bounded Lipschitz domain.
    For any $\alpha>0$, the operator $(-\Delta)^{-\alpha}$ has a positive, positive definite kernel given by
    \begin{align}\label{G}
        G_{\alpha}(x,y) = \frac{1}{\Gamma(\alpha)} \int_0^\infty r^{\alpha-1} P_r(x,y) \, \d r \quad \forall x, y \in D,
    \end{align}
    in the sense that for any $\phi \in \mathcal{S}(D)$,
    \begin{align}\label{frac:Lapl:ker}
        (-\Delta)^{-\alpha} \phi(x) = \int_D G_\alpha(x,y) \phi(y) \, \d y, \quad \forall x \in D.
    \end{align}
    In particular, for any $\phi, \psi \in \mathcal{S}(D)$, we have
    \begin{align}\label{G:id}
        \big\langle (-\Delta)^{-\alpha/2}\phi \,, (-\Delta)^{-\alpha/2}\psi \big\rangle_{L^2(D)} = \int_D \int_D \phi(y) G_{\alpha}(y,z) \psi(z) \, \d y\, \d z.
    \end{align}   
\end{lemma}

\begin{proof}
Lemma \ref{lem:P} implies that the integral in \eqref{G} is convergent and $G_\alpha(x,y)>0$ for any distinct $x,y \in D$.
The relation \eqref{frac:Lapl:ker} follows from \eqref{G}, \eqref{P}, and the identity 
$\lambda^{-\alpha} = \frac{1}{\Gamma(\alpha)}\int_0^\infty r^{\alpha-1} e^{-\lambda r}\, \d r$ for $\alpha,\lambda>0$.
As the kernel of the positive operator $(-\Delta)^{-\alpha}$, $G_\alpha(x,y)$ is positive definite.
Finally, \eqref{G:id} follows from \eqref{frac:Lapl} and the orthonormality of $\{f_n\}_{n\ge1}$.
\end{proof}

\begin{lemma}\label{lem:G:bd}
Let $D \subset \R^d$ be a bounded Lipschitz domain.
For any $\alpha>0$, there exist $C,C'>0$ such that
\[
	C f_1(x)f_1(y) K_{d,\alpha}(|x-y|) \le  G_\alpha(x,y) \le C' K_{d,\alpha}(|x-y|)
\]
for all $x,y \in D$, where $f_1$ is the first Dirichlet eigenfunction and
\[
	K_{d,\alpha}(|x-y|) = \begin{cases}
	|x-y|^{2\alpha-d}& \text{if $0<\alpha<d/2$,}\\
	\log_+(1/|x-y|) & \text{if $\alpha=d/2$,}\\
	1 & \text{if $\alpha>d/2$.}
	\end{cases}
\]
\end{lemma}

\begin{proof}
Let $M$ be the diameter of $D$.
By \eqref{G} and the upper bound in Lemma \ref{lem:P},
\begin{align*}
	&G_\alpha(x,y) 
	\lesssim \int_0^\infty \frac{r^{\alpha-1} e^{-\lambda_1 r}}{1\wedge r^{d/2}} e^{-\frac{|x-y|^2}{C_2 r}}\, \d r\\
	& \lesssim \int_0^{|x-y|^2} r^{\alpha-d/2-1} e^{-\frac{|x-y|^2}{C_2r}} \, \d r + \int_{|x-y|^2}^{2M^2} r^{\alpha-d/2-1} \, \d r + \int_{2M^2}^\infty r^{\alpha-1} e^{-\lambda_1 r} \, \d r.
\end{align*}
For the first term, we may use the fact that $\sup_{0<z \le 1} z^{\alpha-d/2-1}e^{-1/z} < \infty$. Then, it follows that
\[
	G_\alpha(x,y) \lesssim |x-y|^{2\alpha-d} + K_{d,\alpha}(|x-y|) + 1 \lesssim K_{d,\alpha}(|x-y|).
\]
To show the lower bound, we use the lower bound in Lemma \ref{lem:P} and deduce as follows:
\begin{align*}
	G_\alpha(x,y) &\gtrsim (1\wedge f_1(x))(1\wedge f_1(y)) \int_0^\infty \frac{r^{\alpha-1} e^{-\lambda_1 r}}{1\wedge r^{d/2}} e^{-\frac{|x-y|^2}{C_1 r}}\, \d r\\
	& \gtrsim f_1(x)f_1(y) \int_{|x-y|^2}^{2M^2} r^{\alpha-d/2-1} \, \d r \gtrsim f_1(x) f_1(y) K_{d,\alpha}(|x-y|).
\end{align*}
This completes the proof.
\end{proof}

\subsection{Gaussian noise and stochastic integrals}

Let $\alpha > 0$.
The Gaussian noise $W_\alpha$ in \eqref{eq:SHE} is defined as follows.
Define the Hilbert space $\mathcal{H}^\alpha(D)$ as the completion of $\mathcal{S}(D)$ with respect to the inner product
\begin{align}\begin{split}\label{ip:alpha}
    &\langle \phi\,, \psi \rangle_{\alpha}
    := \sum_{n=1}^\infty \lambda_n^{-\alpha} \langle \phi\,, f_n \rangle_{L^2(D)} \langle \psi\,, f_n \rangle_{L^2(D)}\\
    &= \big\langle (-\Delta)^{-\alpha/2} \phi\,, (-\Delta)^{-\alpha/2} \psi \big\rangle_{L^2(D)}
    = \int_D \int_D \phi(y) G_{\alpha}(y,z) \psi(z) \, \d y \, \d z,
\end{split}\end{align}
where the last two identities are due to \eqref{frac:Lapl} and \eqref{G:id}.
Next, define the Hilbert space
\begin{equation*}\label{eq:Hilbertspace}
    \mathcal{H}_{\alpha}(D) = L^2(\R_+\,; \mathcal{H}^{\alpha}(D)) = \left\{ f: \R_+ \to \mathcal{H}^\alpha(D) : \int_0^\infty \|f(s\,,\cdot)\|_{\alpha}^2\, \d s < \infty \right\}.
\end{equation*}
The noise $W_\alpha =\{ W_{\alpha}(\phi)\,; \phi \in \mathcal{H}_{\alpha}(D) \}$ is defined as a centered isonormal Gaussian process on a complete probability space $(\Omega, \mathcal{F}, \P)$ with covariance
\begin{align}\begin{split}\label{W:corr:WfWg}
    &\E[W_{\alpha}(f)W_{\alpha}(g)] = \int_0^\infty \langle f(s,\cdot), g(s,\cdot) \rangle_{\alpha} \, \d s\\
    &= \frac{1}{\Gamma(\alpha)} \int_0^\infty \d s \int_0^\infty \d r \, r^{\alpha-1} \int_D \int_D \d y\, \d z \, f(s,y) P_r(y,z) g(s,z)
\end{split}\end{align}
for any $f, g \in \mathcal{H}_\alpha(D)$, where the last equality follows from \eqref{G} and \eqref{ip:alpha}.
For any $t>0$ and $x \in D$, define $p_{t,x}$ by
\begin{align}\label{p}
    p_{t,x}(s,y) = P_{t-s}(x,y) \1_{(0,t)}(s) \qquad \forall s \in (0,t), y \in D.
\end{align}
Hence, \eqref{eq:mild} can be interpreted as $u(t,x) = W_\alpha(p_{t,x})$ whenever it is well defined.

\begin{lemma}\label{lem:p:S(D)}
For any $t>s>0$ and $x \in D$, $p_{t,x}(s,\cdot) \in \mathcal{S}(D)$.
\end{lemma}

\begin{proof}
Fix $t>s>0$ and $x \in D$.
By \eqref{P} and orthonormality of $\{f_n\}_{n \ge 1}$, for any $k \ge 0$,
\begin{align*}
	n^k \left|\langle p_{t,x}(s,\cdot), f_n \rangle_{L^2(D)}\right|
	& = n^k \left| \langle P_{t-s}(x,\cdot), f_n\rangle_{L^2(D)}\right|
	= n^k |e^{-\lambda_n(t-s)} f_n(x)|\\
	&\lesssim n^k e^{-\lambda_n(t-s)} \lambda_n^{d/4} \to 0 \quad \text{as $n \to \infty$,}
\end{align*}
where we have used the fact that $\sup_{x \in D}|f_n(x)|\lesssim \lambda_n^{d/4}$ (see \cite[Proposition 5]{F95}) to obtain the last inequality and Weyl's law $\lambda_n \asymp n^{2/d}$ (see \cite[Corollary 3.15]{FLW}) to conclude the limit as $n \to \infty$.
This shows that $p_{t,x}(s,\cdot) \in \mathcal{S}(D)$.
\end{proof}

\section{Existence and variance bounds}\label{s:exist}

In this section, we establish a necessary and sufficient condition for the existence of solution and derive variance bounds.
We start with the following lemma.

\begin{lemma}\label{lem:r-integral}
Fix $\alpha>0$ and $\lambda>0$.
Then, there exists $C>0$ such that
\[
	\int_0^\infty \frac{r^{\alpha-1} e^{-\lambda(2s+r)}}{1 \wedge (2s+r)^{d/2}} \, \d r
	\le \begin{cases}
	C e^{-2\lambda s}(1+ s^{\alpha-d/2}) & \text{if $\alpha<d/2$,}\\
	C e^{-2\lambda s}\log(e + 1/s) & \text{if $\alpha=d/2$,}\\
	C e^{-2\lambda s} & \text{if $\alpha>d/2$,}
	\end{cases}
\]
uniformly for all $s>0$.
\end{lemma}

\begin{proof}
For $0<s<1/2$, we split the integral and estimate as follows:
\[
	\int_0^{2s} \frac{r^{\alpha-1}\d r}{(2s)^{d/2}}  + \int_{2s}^1 r^{\alpha-d/2-1} \, \d r + \int_1^\infty r^{\alpha-1} e^{-\lambda(2s+r)} \, \d r
	\lesssim
	\begin{cases}
	s^{\alpha-d/2} & \text{if $\alpha<d/2$,}\\
	\log(1/s) & \text{if $\alpha=d/2$,}\\
	1 & \text{if $\alpha>d/2$.}
	\end{cases}
\]
For $s\ge 1/2$, we bound the integral by
\[
	\int_0^\infty r^{\alpha-1} e^{-\lambda(2s+r)} \, \d r \propto e^{-2\lambda s}.
\]
Combine the two cases to obtain the asserted bounds.
\end{proof}

\subsection{Proof of Theorem \ref{thm:exist}}

\begin{proof}
We shall prove that $p_{t,x} \in \mathcal{H}_\alpha(D)$ for every $(t,x) \in (0\,,\infty) \times D$ if and only if \eqref{Dalang} holds. That is,
\[
	\int_0^\infty \|p_{t,x}(s,\cdot)\|_\alpha^2 \, \d s < \infty 
	\quad \Leftrightarrow \quad 
	\alpha>d/2-1.
\]
By Lemma \ref{lem:p:S(D)}, for any $t>s>0$ and $x \in D$, $p_{t,x}(s,\cdot) \in \mathcal{S}(D)$, so we may use \eqref{ip:alpha}, \eqref{G}, symmetry of $P_t$, and the semigroup property \eqref{semigroup} to write
\begin{align}
	&\int_0^\infty \|p_{t,x}(s,\cdot)\|_\alpha^2 \, \d s
        = \int_0^t \int_D \int_D P_{t-s}(x,y) G_\alpha(y,z) P_{t-s}(x,z) \, \d y \, \d z \, \d s\notag \\
        & = \frac{1}{\Gamma(\alpha)} \int_0^t \d s \int_0^\infty \d r \, r^{\alpha-1} \int_D \int_D \d y\, \d z\, P_s(x,y) P_r(y,z) P_s(z,x)\notag \\
        & = \frac{1}{\Gamma(\alpha)} \int_0^t \d s \int_0^\infty \d r \, r^{\alpha-1} P_{2s+r}(x,x).
        \label{int:p:norm-sq}
\end{align}
Lemmas \ref{lem:P} and \ref{lem:r-integral} imply that
\begin{align}
	\int_{1/4}^t \d s \int_0^\infty \d r \, r^{\alpha-1} P_{2s+r}(x,x) \lesssim \int_{1/4}^t \d s \int_0^\infty \d r \, \frac{r^{\alpha-1} e^{-\lambda_1(2s+r)}}{1 \wedge (2s+r)^{d/2}} =O(1)
\end{align}
uniformly for all $t \ge 1/4$.
Hence, it suffices to prove that
\begin{align}\label{I:finite}
	I=I(t,x) < \infty \quad \forall (t,x)\in (0,1/4)\times D
	\quad \Leftrightarrow \quad
	\alpha>d/2-1,
\end{align}
where
\[
	I=I(t,x):=\int_0^t \d s \int_0^\infty \d r \, r^{\alpha-1} P_{2s+r}(x,x).
\]
First, suppose that $\alpha>d/2-1$.
By Lemmas \ref{lem:P} and \ref{lem:r-integral}, 
\[
	I \lesssim \int_0^t \d s \int_0^\infty \d r \, \frac{r^{\alpha-1} e^{-\lambda_1(2s+r)}}{1 \wedge (2s+r)^{d/2}} \lesssim \begin{cases}
	t^{\alpha-d/2 + 1} & \text{if $d/2-1<\alpha<d/2$,}\\
	t \log_+(1/t) & \text{if $\alpha=d/2$,}\\
	t & \text{if $\alpha>d/2$,}
	\end{cases}
\]
uniformly for all $t \in (0,1/4)$. 
This proves that $I<\infty$.

Conversely, suppose that $I<\infty$. We may use the lower bound in Lemma \ref{lem:P} to deduce that
\begin{align*}
	I &\gtrsim (1\wedge f_1(x))^2 \int_0^t \d s \int_0^{2s} \d r \, \frac{r^{\alpha-1}\e^{-\lambda_1(2s+r)}}{(2s+r)^{d/2}}\\
        &\gtrsim |f_1(x)|^2\, \e^{-\lambda_1} \int_0^t \d s\, s^{-d/2} \int_0^{2s} \d r \, r^{\alpha-1}
        \gtrsim |f_1(x)|^2\, \e^{-\lambda_1}\int_0^t s^{\alpha-d/2} \, \d s,
\end{align*}
uniformly for all $t \in (0,1/4)$ and $x \in D$.
Since $I<\infty$, the last integral $\int_0^t s^{\alpha-d/2} \, \d s$ is finite, which implies that $\alpha>d/2-1$.
This proves \eqref{I:finite} and completes the proof of Theorem \ref{thm:exist}.
\end{proof}

\subsection{Variance bounds}

The proof of Theorem \ref{thm:exist} implies the following estimate:

\begin{lemma}\label{lem:var:u:1}
    Suppose $D \subset \R^d$ is a bounded Lipschitz domain. If $\alpha>d/2-1$, then 
\[
	\Var(u(t,x)) \lesssim 1 \wedge (\rho_0(t))^2
\]
uniformly for all $t>0$ and $x \in D$, where
\begin{align}\label{rho_0}
	\rho_0(t)
	:= \begin{cases}
	t^{(2\alpha-d+2)/4} & \text{if $d/2-1<\alpha<d/2$,}\\
	\sqrt{t \log_+(1/t)}& \text{if $\alpha = d/2$,}\\
	\sqrt{t} & \text{if $\alpha > d/2$.}
	\end{cases}
\end{align}
\end{lemma}


Next, we establish the dependence of $\Var(u(t,x))$ on the distance-to-boundary $d(x)$ defined by \eqref{d(x)}.

\begin{lemma}\label{lem:var:u:2}
Suppose $D \subset \R^d$ is a bounded $C^{1,\gamma}$ domain for some $\gamma>0$. If $\alpha>d/2-1$, then 
\[
	\Var(u(t,x)) \lesssim (d_0(x))^2
\]
uniformly for all $t>0$ and $x \in D$, where
\begin{align}\label{d_0}
d_0(x):=\begin{cases}
(d(x))^{(2\alpha-d+2)/2}  & \text{if $d/2-1<\alpha<d/2$,}\\
d(x)\sqrt{\log_+(1/d(x))} & \text{if $\alpha = d/2$,}\\
d(x)& \text{if $\alpha > d/2$.}
\end{cases}
\end{align}
\end{lemma}

\begin{proof}
Note that the case that $d(x) \ge 1$ is trivial because Lemma \ref{lem:var:u:1} implies that
\[
	\Var(u(t,x)) \lesssim 1 \lesssim (d_0(x))^2,
\]
uniformly for all $t>0$ and $x \in D$ with $d(x) \ge 1$.
Next, we consider the case that $d(x)<1$.
By \eqref{W:corr:WfWg}, \eqref{int:p:norm-sq}, the change of variable $r \to 2s+r$ and Fubini's theorem,
\begin{align*}
	\Var(u(t,x)) 
	&\propto \int_0^t \d s \int_0^\infty \d r \, r^{\alpha-1} P_{2s+r}(x,x)\\
	&= \int_0^t \d s \int_{2s}^\infty \d r \, (r-2s)^{\alpha-1} P_r(x,x)\\
	&= \int_0^\infty \left[\int_0^{t \wedge (r/2)} (r-2s)^{\alpha-1}  \d s \right] P_r(x,x) \, \d r
	\lesssim \int_0^\infty r^{\alpha} P_r(x,x) \, \d r,
\end{align*}
uniformly for all $t>0$ and $x \in D$.
Then, we may use Lemma \ref{lem:P} and the inequality $e^{-\lambda_1 r}/(1 \wedge r^{d/2}) \lesssim r^{-d/2}e^{-\frac12 \lambda_1 r}$ for all $r>0$ to deduce the following:
\begin{align*}
	&\Var(u(t,x))
	\lesssim \int_0^\infty \left( 1 \wedge \frac{d(x)}{1\wedge \sqrt{r}} \right)^2 r^{\alpha-d/2} e^{-\frac12\lambda_1 r} \, \d r\\
	& \lesssim \int_0^{d^2(x)} r^{\alpha-d/2} \, \d r + d^2(x) \int_{d^2(x)}^1 r^{\alpha-d/2-1} \, \d r + d^2(x)\int_1^\infty r^{\alpha-d/2} e^{-\frac12\lambda_1 r} \, \d r\\
	&\lesssim (d_0(x))^2
\end{align*}
uniformly for all $t>0$ and $x \in D$ with $d(x)<1$.
This completes the proof.
\end{proof}

Combining Lemmas \ref{lem:var:u:1} and \ref{lem:var:u:2} yields the following:

\begin{proposition}\label{prop:var:u}
Suppose $D \subset \R^d$ is a bounded $C^{1,\gamma}$ domain for some $\gamma>0$.
If $\alpha>d/2-1$, then
\[
	\Var(u(t,x)) \lesssim (d_0(x) \wedge \rho_0(t))^2
\]
uniformly for all $t>0$ and $x \in D$, where $\rho_0$ and $d_0$ are given by \eqref{rho_0} and \eqref{d_0}, respectively.
\end{proposition}

\subsection{A series representation}

\begin{proposition}\label{prop:spec:rep}
Suppose $D \subset \R^d$ is a bounded Lipschitz domain. 
Then, under condition \eqref{Dalang}, the solution has the representation
\begin{equation}\label{eq:spectral}
    u(t,x)=
    \sum_{n=1}^\infty
    \lambda_n^{-\alpha/2}
    \left(\int_0^t \e^{-\lambda_n(t-r)}\,\d B_n(r)\right)f_n(x),
\end{equation}
where \(\{B_n(t)\}_{n\ge1}\) are independent standard Brownian motions.
\end{proposition}

\begin{proof}
Note that both sides of \eqref{eq:spectral} are centered Gaussian processes, so it suffices to show that they have the same covariance function.
The covariance of $u(t,x)$ can be computed as follows, using \eqref{eq:mild}, the first identities in both \eqref{ip:alpha} and \eqref{W:corr:WfWg}, and \eqref{P}:
\begin{align*}
	\E[u(t,x)u(s,y)]
	&=\int_0^{t \wedge s} \langle P_{t-r}(x, \cdot) , P_{s-r}(y,\cdot) \rangle_\alpha\, \d r\\
	&=\sum_{n=1}^\infty\lambda_n^{-\alpha}f_n(x)f_n(y)\int_0^{t\wedge s} \e^{-\lambda_n(t-r)}\e^{-\lambda_n(s-r)}\,\d r.
\end{align*}
The last expression coincides with the covariance of the right-hand side of \eqref{eq:spectral} by independence of $B_n$ and Wiener isometry.
Hence, the Gaussian processes on both sides of \eqref{eq:spectral} have the same law.
\end{proof}

\section{H\"older regularity}\label{sec:regularity}

In this section, we study temporal and spatial H\"older regularity of the solution. 
We first establish some lemmas.

\begin{lemma}\label{lem:spectral:1}
For any $a \in \R$, there exists $C>0$ such that
\begin{align*}
	\sum_{n \ge 1: \lambda_n \le R} \lambda_n^{-a} |f_n(x)|^2 \le 
	\begin{cases}
	C(1+R^{d/2-a}) & \text{if $a<d/2$,}\\
	C\log_+ R & \text{if $a=d/2$,}\\
	C & \text{if $a>d/2$,}
	\end{cases}
\end{align*}
uniformly for all $x \in D$ and $R > 0$.
\end{lemma}

\begin{proof}
First consider $a=0$. By Lemma \ref{lem:P},
\[
	\sum_{n: \lambda_n \le R} |f_n(x)|^2 \le e \sum_{n: \lambda_n \le R} e^{-\lambda_n/R}|f_n(x)|^2
	\le e \, P_{1/R}(x,x) \lesssim R^{d/2}.
\]
Next, we let $k_0 \in \mathbb{Z}$ be such that $2^{k_0}< \lambda_1\le 2^{k_0+1}$.
Then, for any $a \in \R$, a dyadic decomposition, together with the preceding, implies that
\begin{align*}
	\sum_{n \ge 1: \lambda_n \le R} \lambda_n^{-a} |f_n(x)|^2 
	&\le \sum_{k\ge k_0: 2^k \le R} 2^{-a k}\sum_{n\ge 1: 2^{k}< \lambda_n \le 2^{k+1}} |f_n(x)|^2\\
	&\lesssim \sum_{k \ge k_0: 2^k \le R} 2^{(d/2-a)k}
	\lesssim 
	\begin{cases}
	1+R^{d/2-a} & \text{if $a<d/2$,}\\
	\log_+ R & \text{if $a=d/2$,}\\
	1 & \text{if $a>d/2$.}
	\end{cases}
\end{align*}
This completes the proof.
\end{proof}

\begin{lemma}\label{lem:spectral:2}
For any $a > d/2$, there exists $C>0$ such that
\begin{align*}
	\sum_{n \ge 1: \lambda_n \ge R} \lambda_n^{-a} |f_n(x)|^2 \le C R^{d/2-a}
\end{align*}
uniformly for all $x \in D$ and $R>0$.
\end{lemma}

\begin{proof}
Let $k_1 \in \mathbb{Z}$ be such that $2^{k_1} \le R < 2^{k_1+1}$. Then by Lemma \ref{lem:spectral:1},
\begin{align*}
	\sum_{n \ge 1: \lambda_n \ge R} \lambda_n^{-a} |f_n(x)|^2 
	&\le \sum_{k= k_1}^\infty 2^{-a k}\sum_{n\ge 1: 2^{k}< \lambda_n \le 2^{k+1}} |f_n(x)|^2\\
	&\lesssim \sum_{k=k_1}^\infty 2^{(d/2-a)k}
	\lesssim 2^{(d/2-a)k_1}
	\lesssim R^{d/2-a}.
\end{align*}
This completes the proof.
\end{proof}

\begin{lemma}\label{lem:spectral:3}
Suppose \(D\subset\R^d\) is a bounded \(C^2\) domain.
Then, for any $a<d/2+1$, there exists $C>0$ such that
\[
    \sum_{n \ge 1: \lambda_n\le R}\lambda_n^{-a}|\nabla f_n(x)|^2
    \le CR^{d/2+1-a}
\]
uniformly for all $x \in D$ and $R > 0$.
\end{lemma}

\begin{proof}
By \eqref{P} and Lemma \ref{lem:grad:P}, for any $R > 0$,
\[
\begin{aligned}
    \sum_{n \ge 1}\e^{-2\lambda_n/R}|\nabla f_n(x)|^2
    &=
    \int_D|\nabla_x P_{R^{-1}}(x,y)|^2\,\d y        \\
    &\lesssim
    R^{d+1}
    \int_D
    \exp\left(-\frac{2|x-y|^2}{C/R}\right)\,\d y    
    \lesssim R^{d/2+1}.
\end{aligned}
\]
Therefore
\[
    \sum_{n \ge 1: \lambda_n\le R}|\nabla f_n(x)|^2
    \le
    \e^2\sum_{n \ge 1} \e^{-2\lambda_n/R}|\nabla f_n(x)|^2
    \lesssim R^{d/2+1}.
\]
Let $k_0 \in \mathbb{Z}$ be such that $2^{k_0} < \lambda_1 \le 2^{k_0+1}$.
Then, for any $a<d/2+1$, applying the same dyadic decomposition as in the proof of Lemma \ref{lem:spectral:1} gives
\[
\begin{aligned}
    \sum_{n \ge 1: \lambda_n\le R}
    \lambda_n^{-a}|\nabla f_n(x)|^2
    &\le \sum_{k \ge k_0: 2^k \le R} 2^{-ak} \sum_{n \ge 1: 2^k < \lambda_n \le 2^{k+1}} |\nabla f_n(x)|^2\\
    & \lesssim \sum_{k \ge k_0: 2^k \le R} 2^{(d/2+1-a)k}  \lesssim R^{d/2+1-a}.
\end{aligned}
\]
This completes the proof of Lemma \ref{lem:spectral:3}.
\end{proof}

\subsection{Temporal regularity}


\begin{proposition}\label{prop:u-u:t}
Suppose $D \subset \R^d$ is a bounded Lipschitz domain and $\alpha>d/2-1$.
Then, for any $T>0$, there exists $C>0$ such that
\[
	\Var(u(t',x)-u(t,x)) \le C \rho_0^2(|t'-t|)
\]
uniformly for all $t,t' \in [0,T]$ and $x \in D$, where $\rho_0$ is given by \eqref{rho_0}, i.e.,
\[
	\rho_0^2(r) = \begin{cases}
	r^{\alpha-d/2+1} & \text{if $d/2-1<\alpha<d/2$,}\\
	r \log_+(1/r) & \text{if $\alpha=d/2$,}\\
	r & \text{if $\alpha>d/2$.}
	\end{cases}
\]
\end{proposition}

\begin{proof}
For $0\le t < t' \le T$ and $x \in D$, we may use \eqref{W:corr:WfWg} and the semigroup property \eqref{semigroup} to write $\Var(u(t',x)-u(t,x))\lesssim I_1+I_2$, where
\begin{align*}
	& I_1:=\int_t^{t'}\d s\int_0^\infty\d r\,
    r^{\alpha-1}P_{2(t'-s)+r}(x,x) = \int_0^{t'-t} \d s \int_0^\infty \d r \, r^{\alpha-1} P_{2s+r}(x,x),\\
	&I_2:=\int_0^t \d s \int_0^\infty \d r\, r^{\alpha-1} \big[P_{2s+2(t'-t)+r}(x,x)-2P_{2s+(t'-t)+r}(x,x) 
        +P_{2s+r}(x,x) \big].
\end{align*}
By Lemma \ref{lem:var:u:1},
\begin{align}\label{var:u-u:t:I1}
	I_1 = \Var(u(t'-t, x)) \lesssim 
	\begin{cases}
	|t'-t|^{\alpha-d/2+1} & \text{if $d/2-1<\alpha<d/2$,}\\
	|t'-t| \log_+(1/|t'-t|) & \text{if $\alpha=d/2$,}\\
	|t'-t| & \text{if $\alpha>d/2$.}
	\end{cases}
\end{align}
Next, we may use \eqref{P} and Fubini's theorem to write
\begin{align*}
	I_2 &= \int_0^t \d s \int_0^\infty \d r \, r^{\alpha-1} \sum_{n=1}^\infty e^{-\lambda_n(2s+r)} \left( 1 - e^{-\lambda_n(t'-t)} \right)^2 |f_n(x)|^2\\
	& = \sum_{n=1}^\infty \int_0^\infty \d r \, r^{\alpha-1} e^{-\lambda_n r} \left( \frac{1-e^{-2\lambda_n t}}{2\lambda_n} \right) \left( 1 - e^{-\lambda_n(t'-t)} \right)^2 |f_n(x)|^2.
\end{align*}
Then, the identity $\int_0^\infty r^{\alpha-1} e^{-\lambda r} \, \d r = \Gamma(\alpha) \lambda^{-\alpha}$ for $\alpha,\lambda>0$ and the inequality $1-e^{-z} \le 1 \wedge z$ for $z>0$ imply that
\begin{align*}
	I_2 &\lesssim \sum_{n=1}^\infty \lambda_n^{-\alpha-1} \left( 1 \wedge \lambda_n^2|t'-t|^2 \right) |f_n(x)|^2\\
	&\lesssim |t'-t|^2 \sum_{n \ge 1: \lambda_n \le 1/|t'-t|} \lambda_n^{-\alpha+1} |f_n(x)|^2 + \sum_{n \ge 1: \lambda_n > 1/|t'-t|} \lambda_n^{-\alpha-1} |f_n(x)|^2\\
	& \lesssim \begin{cases}
	|t'-t|^2 |t'-t|^{-d/2+\alpha-1} + |t'-t|^{-d/2+\alpha+1} & \text{if $d/2-1<\alpha<d/2+1$,}\\
	|t'-t|^2 \log_+(1/|t'-t|) + |t'-t|^{-d/2+\alpha+1} & \text{if $\alpha=d/2+1$,}\\
	|t'-t|^2  + |t'-t|^{-d/2+\alpha+1} & \text{if $\alpha>d/2+1$.}
	\end{cases}
\end{align*}
It follows that
\begin{align}\label{var:u-u:t:I2}
	I_2 \lesssim 
	\begin{cases}
	|t'-t|^{\alpha-d/2+1} & \text{if $d/2-1<\alpha<d/2+1$,}\\
	|t'-t|^2 \log_+(1/|t'-t|) & \text{if $\alpha=d/2+1$,}\\
	|t'-t|^2 & \text{if $\alpha>d/2+1$.}
	\end{cases}
\end{align}
Combining \eqref{var:u-u:t:I1} and \eqref{var:u-u:t:I2} yields the desired estimate.
\end{proof}

\subsection{Spatial regularity}

In order to establish spatial regularity, we will 
need the following lemma, whose proof will be given after the proof of Proposition \ref{prop:u-u:x}.

\begin{lemma}\label{lem:curve}
If $D\subset \R^d$ is a bounded $C^k$ domain for some $k \in \N_+$, then there exist $\varepsilon_0 \in (0,1)$ and $C_0>0$ such that the following property holds: For every pair $x,y \in D$ with $|x-y|\le \varepsilon_0$, there is a $C^k$ curve $\gamma:[0,1] \to D$ such that $\gamma(0)=x$, $\gamma(1)=y$, and
\[
	|\gamma'(\tau)| \le C_0|x-y| \qquad \forall \tau \in [0,1].
\]
\end{lemma}

Assuming Lemma \ref{lem:curve}, we can prove spatial regularity of the solution.

\begin{proposition}\label{prop:u-u:x}
Suppose $D \subset \R^d$ is a bounded $C^2$ domain and $\alpha>d/2-1$.
Then, for any $T>0$, there exists $C>0$ such that
\begin{align}\label{var:u-u:x}
	\Var(u(t,x')-u(t,x)) \le C \rho_1^2(|x'-x|),
\end{align}
uniformly for all $t \in [0,T]$ and $x,x'\in D$, where $\rho_1$ is given by
\[
	\rho_1^2(r) :=\begin{cases}
	r^{2\alpha-d+2} & \text{if $d/2-1<\alpha<d/2$,}\\
	r^2 \log_+(1/r) & \text{if $\alpha=d/2$,}\\
	r^2 & \text{if $\alpha>d/2$,}
	\end{cases}
\]
\end{proposition}

\begin{proof}
Let $\varepsilon_0 \in (0,1)$ and $C_0>0$ be the constants given by Lemma \ref{lem:curve}.
According to Lemma \ref{lem:var:u:1}, $\Var(u(t,x))$ is uniformly bounded, so it suffices to prove \eqref{var:u-u:x} for $|x'-x| \le \varepsilon_0$.
For any $t \in [0,T]$ and $x,x' \in D$ with $|x'-x|\le \varepsilon_0$, we may use \eqref{W:corr:WfWg}, the semigroup property \eqref{semigroup} and \eqref{P} to see that
\begin{align*}
	&\Var(u(t,x')-u(t,x))\\
	&\propto \int_0^t \d s \int_0^\infty \d r \, r^{\alpha-1} \left[ P_{2s+r}(x',x') - 2P_{2s+r}(x',x)+ P_{2s+r}(x,x) \right]\\
	&=\int_0^t \d s \int_0^\infty \d r \, r^{\alpha-1} \sum_{n=1}^\infty e^{-\lambda_n (2s+r)} |f_n(x')-f_n(x)|^2\\
	&=\int_0^t \d s \int_0^\infty \d r \, r^{\alpha-1} \int_D \d y\, |P_{s+r/2}(x',y)-P_{s+r/2}(x,y)|^2.
\end{align*}
On the one hand, we may use $(a+b)^2 \le 2(a^2+b^2)$ for $a,b \ge 0$ and Lemma \ref{lem:P} to deduce that
\begin{align*}
	&\int_D |P_{s+r/2}(x',y)-P_{s+r/2}(x,y)|^2 \, \d y\\
	& \lesssim \int_D \frac{e^{-2\lambda_1 (s+r/2)}}{1 \wedge (s+r/2)^d} \, e^{- \frac{2|x-y|^2}{C_2(s+r/2)}} \, \d y + \int_D \frac{e^{-2\lambda_1 (s+r/2)}}{1\wedge (s+r/2)^d} e^{- \frac{2|x'-y|^2}{C_2(s+r/2)}}  \, \d y\\
	& \lesssim \int_{\R^d} \frac{e^{-\lambda_1 (s+r/2)}}{(s+r/2)^d} \, e^{- \frac{2|x-y|^2}{C_2(s+r/2)}} \, \d y + \int_{\R^d} \frac{e^{-\lambda_1 (s+r/2)}}{ (s+r/2)^d} e^{- \frac{2|x'-y|^2}{C_2(s+r/2)}}  \, \d y
	\lesssim \frac{e^{-\frac12\lambda_1(2s+r)}}{(2s+r)^{d/2}}
\end{align*}
uniformly for all $s \in (0,T]$, $r > 0$ and $x,x'\in D$.
On the other hand, by Lemma \ref{lem:curve}, for any $x,x'\in D$, we can find a $C^2$ curve $\gamma:[0,1] \to D$ going from $x$ to $x'$ such that $|\gamma'(\tau)| \le C_0|x'-x|$ for any $\tau \in [0,1]$.
Then, by the fundamental theorem of calculus, Cauchy-Schwarz inequality and Lemma \ref{lem:grad:P}, there exists $C>0$ such that
\begin{align*}
	&\int_D |P_{s+r/2}(x',y)-P_{s+r/2}(x,y)|^2 \, \d y
	= \int_D\left| \int_0^1 \nabla_x P_{s+r/2}(\gamma(\tau),y) \cdot \gamma'(\tau) \, \d \tau \right|^2 \,  \d y \\
	& \lesssim \int_D \d y \left( \int_0^1 \d \tau\, |\gamma'(\tau)|^2 \right) \left( \int_0^1 \d \tau  \, |\nabla_x P_{s+r/2}(\gamma(\tau),y)|^2 \right)\\
	& \lesssim |x'-x|^2 \int_0^1 \d \tau \int_{\R^d} \d y \, \frac{e^{-\lambda_1(s+r/2)}}{(s+r/2)^{d+1}} e^{- \frac{2|\gamma(\tau)-y|^2}{C(s+r/2)}}
	\lesssim |x'-x|^2 \frac{e^{-\frac12 \lambda_1 (2s+r)}}{(2s+r)^{d/2+1}}
\end{align*}
uniformly for all $s \in (0,T]$, $r>0$ and $x,x' \in D$ with $|x'-x| \le \varepsilon_0$.
Combining these two bounds, we have
\begin{align*}
	\int_D |P_{s+r/2}(x',y)-P_{s+r/2}(x,y)|^2 \, \d y
	\lesssim
	 \frac{e^{-\frac12 \lambda_1 (2s+r)}}{(2s+r)^{d/2}} \left( 1 \wedge \frac{|x'-x|^2}{2s+r} \right),
\end{align*}
uniformly for all $s \in (0,T]$, $r>0$ and $x,x' \in D$ with $|x'-x| \le \varepsilon_0$. 
It follows that
\begin{align*}
	\Var(u(t,x')-u(t,x))
	& \lesssim \int_0^t \d s \int_0^\infty \d r \,\frac{r^{\alpha-1} e^{-\frac12\lambda_1(2s+r)}}{(2s+r)^{d/2}} \left( 1 \wedge \frac{|x'-x|^2}{2s+r} \right)\\
	&=\int_0^t \d s \int_{2s}^\infty \d r \,\frac{(r-2s)^{\alpha-1} e^{-\frac12\lambda_1r}}{r^{d/2}} \left( 1 \wedge \frac{|x'-x|^2}{r} \right)\\
	&=\int_0^\infty \d r \left[\int_0^{t \wedge (r/2)} (r-2s)^{\alpha-1} \, \d s \right] \frac{e^{-\frac12 \lambda_1r}}{r^{d/2}} \left( 1 \wedge \frac{|x'-x|^2}{r}\right)\\
	& \lesssim \int_0^\infty \d r \, r^{\alpha-d/2} e^{-\frac12 \lambda_1 r} \left( 1 \wedge \frac{|x'-x|^2}{r}\right).
\end{align*}
Next, we split the integral and estimate as follows:
\begin{align*}
	\Var(u(t,x')-u(t,x))
	& \lesssim \int_0^{|x'-x|^2} r^{\alpha-d/2} \, \d r + |x'-x|^2 \int_{|x'-x|^2}^1 r^{\alpha-d/2-1} \, \d r\\
	& \qquad  + |x'-x|^2 \int_1^\infty r^{\alpha-d/2-1} e^{-\frac12 \lambda_1 r} \, \d r\\
	& \lesssim \begin{cases}
	|x'-x|^{2\alpha-d+2} & \text{if $d/2-1<\alpha<d/2$,}\\
	|x'-x|^2 \log_+(1/|x'-x|) & \text{if $\alpha=d/2$,}\\
	|x'-x|^2 & \text{if $\alpha>d/2$.}
	\end{cases}
\end{align*}
This completes the proof of Proposition \ref{prop:u-u:x}, assuming Lemma \ref{lem:curve}.
\end{proof}

\begin{proof}[Proof of Lemma \ref{lem:curve}]
Since $D$ is a $C^k$ domain, for every $q \in \partial D$, we can find a radius $\varepsilon \in (0,1)$, a $C^k$ function $\phi:\R^{d-1} \to \R$, and an orthonormal coordinate system $(z,v)_q$ for $\R^d$, with $z \in \R^{d-1}, v \in \R$ and with origin $(0,0)_q=q$ such that
\begin{align*}
	&D \cap B(q,\varepsilon) = B(q,\varepsilon) \cap \{ (z,v+\phi(z))_q : z \in \R^{d-1}, v>0 \},\\
	&\partial D \cap B(q,\varepsilon) = B(q,\varepsilon) \cap \{ (z, \phi(z))_q : z \in \R^{d-1} \}.
\end{align*}
Since $D$ is bounded, its boundary $\partial D$ is compact, so we can find $q_1,\dots,q_n \in \partial D$ such that $\partial D \subset \bigcup_{i=1}^n B(q_i, \varepsilon_i/12)$, where $\varepsilon_1,\dots,\varepsilon_n \in (0,1)$ are the corresponding radii, $\phi_1,\dots, \phi_n :\R^{d-1} \to \R$ are the corresponding $C^1$ functions, and $(z,v)_i$ are the corresponding coordinate systems.

Since each $\phi_i$ is $C^k$ for some $k \ge 1$,
\begin{align}\label{curve:L}
	L_i := \sup_{|z| \le \varepsilon_i} |\nabla \phi_i(z)| < \infty.
\end{align}
We may take 
\begin{align}\label{curve:eps_0}
	C_0 := 2\left[1+\max_{1 \le i \le n}L_i^2\right]
	\quad \text{and}\quad
	\varepsilon_0 := \min_{1\le i\le n} \left[ \frac{\varepsilon_i}{12} \wedge \frac{\varepsilon_i}{4C_0} \right].
\end{align}
Consider any pair $x, y \in D$ with $|x-y| \le \varepsilon_0$.
Since $\overline{B}(x,\varepsilon_0) \cap \bigcup_{i=1}^n B(q_i, \varepsilon_i/12)$ is either empty or nonempty, and since $\varepsilon_0 \le \varepsilon_i/12$, either one of the following holds:
\begin{enumerate}
\item $\overline{B}(x, \varepsilon_0) \subset D$, or
\item $\overline{B}(x, \varepsilon_0) \subset B(q_i, \varepsilon_i/4)$ for some $i \in \{1,\dots, n\}$.
\end{enumerate}

In case (1), we may take $\gamma$ as the straight line from $x$ to $y$, and get $|\gamma'(\tau)| \le |x-y|$ for all $\tau \in [0,1]$.
In case (2), we may write 
\[
	x = F(z_1,v_1) \quad \text{and} \quad 
	y = F(z_2,v_2)
\]
for some $z_1, z_2 \in \R^{d-1}$ and $v_1,v_2>0$, where $F: \R^{d-1} \times \R \to \R^d$ is given by
\[
	F(z,v) = (z, v+ \phi_i(z))_i.
\]
Define $\gamma: [0,1] \to \R^d$ by
\[
	\gamma(\tau) = (\tau z_2+(1-\tau)z_1 \ ,\  \tau v_2+(1-\tau)v_1 + \phi_i(\tau z_2+(1-\tau)z_1))_i.
\]
Then, $\gamma$ is a $C^k$ curve such that $\gamma(0) = x$, $\gamma(1)=y$, and
\begin{align*}
	\gamma'(\tau) &= ( z_2-z_1\ , \ v_2-v_1 + \nabla \phi_i(\tau z_2+(1-\tau)z_1) \cdot (z_2-z_1) )_i.
\end{align*}
By \eqref{curve:L}, we have
\begin{align}\begin{split}\label{gamma':bd}
	|\gamma'(\tau)| &\le \sqrt{ |z_1-z_2|^2 + 2|v_1-v_2|^2+2L_i^2 |z_1-z_2|^2}\\
	& \le \sqrt{2(1+L_i^2)}\, |(z_1,v_1)-(z_2,v_2)|
\end{split}\end{align}
for all $\tau \in [0,1]$.
Note that $F$ is invertible with $F^{-1}(a,b) = (a,b-\phi_i(a))$ and
\begin{align*}
	|F^{-1}(a_1,b_1) - F^{-1}(a_2,b_2)|
	&\le \sqrt{2(1+L_i^2)}\, |(a_1,b_1)-(a_2,b_2)|
\end{align*}
for all $(a_1,b_1), (a_2,b_2) \in \R^{d-1} \times \R$.
It follows that
\[
	|(z_1,v_1)-(z_2,v_2)| = |F^{-1}(x)-F^{-1}(y)| \le \sqrt{2(1+L_i^2)} \, |x-y|.
\]
The preceding together with \eqref{gamma':bd} and the choice of $C_0$ in \eqref{curve:eps_0} implies that
\begin{align}\label{gamma':bd2}
	|\gamma'(\tau)| \le 2(1+L_i^2) |x-y| \le C_0|x-y| \quad \text{for all $\tau \in [0,1]$.}
\end{align}
This yields the desired estimate.
It remains to verify that the curve $\gamma$ lies in $D$.
Indeed, \eqref{gamma':bd2}, $|x-y| \le \varepsilon_0$, and $\overline{B}(x,\varepsilon_0) \subset B(q_i,\varepsilon_i/4)$ imply that for every $\tau \in [0,1]$, 
\begin{align*}
	|\gamma(\tau)-q_i| &\le |\gamma(\tau)-x| + |x-q_i| = \left|\int_0^\tau \gamma'(r) \, \d r\right| + |x-q_i|\\
	&\le C_0 \varepsilon_0 + \frac{\varepsilon_i}{4}
	\le C_0\left(\frac{\varepsilon_i}{4C_0}\right) + \frac{\varepsilon_i}{4} = \frac{\varepsilon_i}{2},
\end{align*}
where the last inequality follows from the choice of $\varepsilon_0$ in \eqref{curve:eps_0}. This shows that $\gamma(\tau) \in \overline{B}(q_i,\varepsilon_i/2) \cap \{ (z,v+\phi_i(z))_i : z \in \R^{d-1}, v>0 \} \subset D \cap \overline{B}(q_i,\varepsilon_i/2)$ for all $\tau \in [0,1]$.
The proof of Lemma \ref{lem:curve} is thus complete.
\end{proof}

\subsection{Proof of Theorem \ref{thm:reg}}

\begin{proof}
Fix $T>0$.
Under \eqref{Dalang}, i.e., $\alpha>d/2-1$,
Propositions \ref{prop:u-u:t} and \ref{prop:u-u:x} imply that
\begin{align*}
	\|u(t,x)-u(t',x')\|_2
	&\le \|u(t,x)-u(t',x)\|_2 + \|u(t',x)-u(t',x')\|_2\\
	& \lesssim \tilde\rho((t,x),(t',x')),
\end{align*}
uniformly for all $(t,x),(t',x') \in [0,T]\times D$, where 
\[
	\tilde\rho((t,x),(t',x')) :=
	\begin{cases}
	|t-t'|^{\theta/2} + |x-x'|^\theta & \text{if $\frac{d}{2}-1 < \alpha < \frac{d}{2}$,}\smallskip\\
	\sqrt{|t-t'| \log_+\frac{1}{|t-t'|}} + |x-x'|\sqrt{\log_+\frac{1}{|x-x'|}} & \text{if $\alpha=d/2$,}\smallskip\\
	\sqrt{|t-t'|} + |x-x'| & \text{if $\alpha>d/2$,}
	\end{cases}
\]
and $\theta$ is given by \eqref{theta}.
Since $u$ is Gaussian, it follows that for any $k \ge 2$, there exists $C_{k,T} >0$ such that
\[
	\|u(t,x)-u(t',x')\|_k \le C_{k,T} \,\tilde\rho((t,x),(t',x')),
\]
uniformly for all $(t,x),(t',x') \in [0,T]\times D$.
Hence, by Kolmogorov's continuity theorem, $u$ is a.s.~locally H\"older continuous of order $\beta_0$ in $t$ and of order $\beta_1$ in $x$, for any $\beta_0 \in (0,(\theta\wedge 1)/2)$ and $\beta_1 \in (0,\theta\wedge 1)$.
\end{proof}

\section{Optimal regularity}\label{s:optimal}

\subsection{Strong local nondeterminism}


\begin{proposition}\label{prop:SLND}
Suppose \(D\subset\R^d\) is a bounded \(C^2\) domain and \eqref{alpha:Holder:range} holds. Then for every \(T>0\), there exists \(C>0\) such that
\[
    \Var\left(
        u(t,x)\mid u(t_1,x_1),\ldots,u(t_n,x_n)
    \right)
    \ge
    C\left[
        \min_{1\le j\le n}
        \rho^2((t,x),(t_j,x_j))
        \wedge t^{\theta}
        \wedge (d(x))^{2\theta}
    \right]
\]
uniformly for all \(n\ge1\) and all \((t,x),(t_1,x_1),\dots,(t_n,x_n)\in[0,T]\times D\).
\end{proposition}

\begin{proof}
For a centered Gaussian vector $(X,Y_1,\dots, Y_n)$, the conditional variance $\Var(X\mid Y_1,\dots,Y_n)$ is the squared \(L^2(\Omega)\)-distance from $X$ to the linear span of the conditioned variables, i.e.,
\[
	\Var(X\mid Y_1,\dots,Y_n) =\inf_{a_1,\dots, a_n \in \R}\E\Bigg[\bigg(X-\sum_{j=1}^n a_jY_j\bigg)^2\Bigg].
\]
Thus it is enough to prove that there exists $C>0$ such that
\begin{align}\begin{split}\label{SLND:claim}
	&\E\Bigg[\bigg(u(t,x)-\sum_{j=1}^n a_j u(t_j,x_j)\bigg)^2\Bigg]\\
	& \ge C \left[ \min_{1\le j\le n}
        \rho^2((t,x),(t_j,x_j))
        \wedge t^{\theta}
        \wedge (d(x))^{2\theta}\right],
\end{split}\end{align}
uniformly for all $n \ge 1$, $(t,x),(t_1,x_1),\dots, (t_n,x_n) \in [0,T] \times D$, and $a_1,\dots, a_n \in \R$.
To this end, we first recall that
\[
	u(t,x)-\sum_{j=1}^n a_j u(t_j,x_j) = W_\alpha\bigg( p_{t,x} - \sum_{j=1}^n a_j p_{t_j,x_j}\bigg).
\]
By Theorem \ref{thm:exist}, the right-hand side above is a well-defined Wiener integral under condition \eqref{Dalang}.
Hence, we may use the first identities of \eqref{ip:alpha} and \eqref{W:corr:WfWg}, and \eqref{P} to deduce that 
\begin{align}\label{eq:Fourier}
&\E\Bigg[\bigg(u(t,x)-\sum_{j=1}^n a_j u(t_j,x_j)\bigg)^2\Bigg]\\
&= \sum_{k=1}^\infty \lambda_k^{-\alpha} \int_\R \bigg| e^{-\lambda_k (t-r)} f_k(x) {\bf 1}_{(0,t)}(r) - \sum_{j=1}^n a_j e^{-\lambda_k (t_j-r)} f_k(x_j) {\bf 1}_{(0,t_j)}(r) \bigg|^2 \d r \notag\\
&=
    \frac1{2\pi}
    \sum_{k=1}^\infty \lambda_k^{-\alpha}
    \int_\R
    \frac{\left|
        (e^{-i\tau t}-e^{-\lambda_k t})f_k(x)
        -\sum_{j=1}^n a_j
        (e^{-i\tau t_j}-e^{-\lambda_k t_j})f_k(x_j)
    \right|^2}
    {\lambda_k^2+\tau^2}\,\d\tau, \notag
\end{align}
where the last equality follows from Plancherel's theorem and the identity
\[
    \int_0^t \e^{-\lambda(t-r)}\e^{-i\tau r}\,\d r
    =
    \frac{\e^{-i\tau t}-\e^{-\lambda t}}{\lambda-i\tau}.
\]
Let \(R_D=\diam(D)\) and \(A_D=R_D\wedge1\wedge\sqrt T\). 
Define
\begin{align}\label{SLND:r}
    r
    :=
    A_D
    \left[
        \min_{1\le j\le n}
        \left(
            \sqrt{\frac{|t-t_j|}{T}}
            \vee
            \frac{|x-x_j|}{R_D}
        \right)
        \wedge
        \sqrt{\frac{t}{T}}
        \wedge
        \frac{d(x)}{R_D}
    \right].
\end{align}
Choose and fix two test functions
\[
    \varphi\in C_c^\infty((-1/2,1/2)),
    \qquad
    \psi\in C_c^\infty(B(0,1/2)),
\]
with \(\varphi(0)=\psi(0)=1\).
If \(r=0\), the desired lower bound in \eqref{SLND:claim} is trivial. So, we may assume that \(r>0\), and define
\[
    \varphi_{r^2}(u):=r^{-2}\varphi(u/r^2)
    \quad \text{and} \quad
    \psi_{x,r}(y):=r^{-d}\psi\left(\frac{y-x}{r}\right).
\]
Following \cite[Section~3]{HL26}, we define
\begin{align}\label{SLND:I}
    I&:=
    \sum_{k=1}^\infty \inner{\psi_{x,r}}{f_k}_{L^2(D)} \times\\
    & \quad \times
    \int_\R
    \left[
        (\e^{-i\tau t}-\e^{-\lambda_k t})f_k(x)
        -
        \sum_{j=1}^n a_j
        (\e^{-i\tau t_j}-\e^{-\lambda_k t_j})f_k(x_j)
    \right]
    \e^{i\tau t}\widehat\varphi_{r^2}(\tau)\,\d\tau,\notag
\end{align}
where we use the Fourier convention
\[
    \widehat g(\tau)=\int_\R\e^{-i\tau v}g(v)\,\d v,
    \qquad
    g(v)=\frac1{2\pi}\int_\R\e^{i\tau v}\widehat g(\tau)\,\d\tau.
\]
In particular, \(\widehat\varphi_{r^2}(\tau)=\widehat\varphi(r^2\tau)\), and Fourier inversion gives
\begin{align*}
    I
    =2\pi\sum_{k=1}^\infty
    \Bigg[&\bigl(\varphi_{r^2}(0)
        -\e^{-\lambda_k t}\varphi_{r^2}(t)\bigr)f_k(x)\\
    &-\sum_{j=1}^n a_j
       \bigl(\varphi_{r^2}(t-t_j)
        -\e^{-\lambda_k t_j}\varphi_{r^2}(t)\bigr)f_k(x_j)
    \Bigg]
    \inner{\psi_{x,r}}{f_k}_{L^2(D)}.
\end{align*}
Note that \(A_D\le\sqrt T\), thus we have \(r\le\sqrt t\). 
Because \(\operatorname{supp}\varphi\subset[-1/2,1/2]\), the inequality \(r\le\sqrt t\) implies $\varphi_{r^2}(t)=0$.
Hence,
\begin{align*}
	I
	=2\pi\sum_{k=1}^\infty
	\Bigg[\varphi_{r^2}(0) f_k(x)
	-\sum_{j=1}^n a_j\varphi_{r^2}(t-t_j)f_k(x_j)\Bigg]
	\inner{\psi_{x,r}}{f_k}_{L^2(D)}.
\end{align*}
Since \(\psi_{x,r}\in C_c^\infty(D)\), we may use integration by parts, $\sup_{x \in D}|f_n(x)|\lesssim \lambda_n^{d/4}$ (see \cite[Proposition 5]{F95}) and Weyl's law $\lambda_n \asymp n^{2/d}$ (see \cite[Corollary 3.15]{FLW}) to show that for any $z \in D$,
\begin{align*}
	&\sum_{n=1}^\infty \left|\langle \psi_{x,r} , f_n \rangle_{L^2(D)} f_n(z)\right|
	= \sum_{n=1}^\infty \lambda_n^{-d} \left|\langle \psi_{x,r} , (-\Delta)^d f_n \rangle_{L^2(D)}\right| |f_n(z)| \\
	&= \sum_{n=1}^\infty \lambda_n^{-d} \left| \langle(-\Delta)^d \psi_{x,r}, f_n  \rangle_{L^2(D)}\right| |f_n(z)|
	\lesssim \|(-\Delta)^d \psi_{x,r}\|_{L^2(D)} \sum_{n=1}^\infty n^{-3/2} < \infty.
\end{align*}
The preceding and the fact that each $f_n$ is continuous on $D$ (see \cite[Theorem 9.31]{Brezis}) then imply pointwise convergence of the following series:
\[
    \sum_{k=1}^\infty
    \inner{\psi_{x,r}}{f_k}_{L^2(D)}f_k(z)
    =\psi_{x,r}(z) \quad \text{for all $z\in D$.}
\]
It follows that
\begin{align}\label{SLND:I:eval}
    I=2\pi
    \Bigg[\varphi_{r^2}(0)\psi_{x,r}(x)-\sum_{j=1}^n a_j \varphi_{r^2}(t-t_j)\psi_{x,r}(x_j)\Bigg].
\end{align}
Next, we evaluate $I$ by checking the support of the test functions.
Recall \eqref{d(x)}.
Since \(A_D\le R_D\), we have \(r\le d(x)\), so $\psi_{x,r}(\cdot)$ is supported in \(B(x,r/2)\subset D\). 
Fix \(j\in\{1,\ldots,n\}\) and consider the product $\varphi_{r^2}(t-t_j)\psi_{x,r}(x_j)$.
If \(\varphi_{r^2}(t-t_j)=0\), then this product is 0. Otherwise \(|t-t_j|\le r^2/2\). Hence
\[
    \sqrt{\frac{|t-t_j|}{T}}\le \frac{r}{\sqrt{2T}}.
\]
By the definition of \(r\), the maximum
\[
    \sqrt{\frac{|t-t_j|}{T}}\vee\frac{|x-x_j|}{R_D}
\]
is at least \(r/A_D\).
But the first term $\sqrt{|t-t_j|/T}$ in the maximum is at most \(r/\sqrt{2T}<r/A_D\) because \(A_D\le\sqrt T\), so the maximum must be attained by the second term $|x-x_j|/R_D$. Hence,
\[
    r\le A_D\frac{|x-x_j|}{R_D}\le |x-x_j|.
\]
Since \(\operatorname{supp}\psi_{x,r}\subset B(x,r/2)\), this gives \(\psi_{x,r}(x_j)=0\). Therefore, in any case, we have
\begin{equation}\label{eq:vanish}
    \varphi_{r^2}(t-t_j)\psi_{x,r}(x_j)=0
    \quad \text{for each $j=1,\dots, n$.}
\end{equation}
It follows from \eqref{SLND:I:eval} and \eqref{eq:vanish} that
\begin{align}\label{SLND:I:value}
    I
    =2\pi\varphi_{r^2}(0)\psi_{x,r}(x)
    =2\pi r^{-d-2}.
\end{align}

Applying Cauchy-Schwarz directly to \eqref{SLND:I}, with weight
\(\lambda_k^{-\alpha}(\lambda_k^2+\tau^2)^{-1}\), and using
\eqref{eq:Fourier}, gives
\begin{align}\label{SLND:I:CS}
    |I|^2
    \le
    C\,\E\Bigg[\bigg(u(t,x)-\sum_{j=1}^n a_j u(t_j,x_j)\bigg)^2\Bigg]\times J,
\end{align}
where
\[
    J
    :=
    \sum_{k=1}^\infty
    \int_\R
    \lambda_k^\alpha(\lambda_k^2+\tau^2)
    \left|\widehat\varphi_{r^2}(\tau)\right|^2
    \left|\inner{\psi_{x,r}}{f_k}_{L^2(D)}\right|^2\,\d\tau .
\]
Since \(\widehat\varphi_{r^2}(\tau)=\widehat\varphi(r^2\tau)\) and $\widehat{\varphi}$ is rapidly
decreasing, we have
\[
    \int_\R \left|\widehat\varphi_{r^2}(\tau)\right|^2\,\d\tau\lesssim r^{-2},
    \qquad
    \int_\R\tau^2\left|\widehat\varphi_{r^2}(\tau)\right|^2\,\d\tau\lesssim r^{-6}.
\]
Consequently,
\[
	J \lesssim r^{-2} \sum_{k=1}^\infty \lambda_k^{\alpha+2} \left| \langle \psi_{x,r}, f_k \rangle_{L^2(D)} \right|^2 + r^{-6} \sum_{k=1}^\infty \lambda_k^{\alpha} \left| \langle \psi_{x,r}, f_k \rangle_{L^2(D)} \right|^2.
\]
Next, we show that for any $\beta\ge 0$, there exists $C=C_{\beta, \psi}>0$ such that
\begin{align}\label{SLND:claim:J}
	\sum_{k=1}^\infty \lambda_k^{\beta} \left| \langle \psi_{x,r}, f_k \rangle_{L^2(D)} \right|^2 \le C r^{-d-2\beta}.
\end{align}

Case 1: $\beta = m \in \N_0$.
In this case, we may use $-\Delta f_k = \lambda_k f_k$, integration by parts, and the definition of $\psi_{x,r}$ to deduce that
\begin{align*}
	&\langle \psi_{x,r} , \lambda_k^m f_k\rangle
	= \int_D \psi_{x,r}(y) ((-\Delta)^m f_k) (y) \, \d y
	=\int_D ((-\Delta)^m \psi_{x,r}) (y) f_k(y)  \, \d y\\
	& = (-1)^m r^{-2m} \int_D r^{-d}(\Delta^m\psi)(\tfrac{y-x}{r}) f_k(y) \, \d y = (-1)^m r^{-2m} \langle (\Delta^m\psi)_{x,r} , f_k\rangle_{L^2(D)}.
\end{align*}
The preceding, Parseval's identity, and the change of variable $z=(y-x)/r$ then imply that
\begin{align*}
	&\sum_{k=1}^\infty \lambda_k^{m} \left| \langle \psi_{x,r}, f_k \rangle_{L^2(D)} \right|^2
	= \sum_{k=1}^\infty \langle \psi_{x,r}, f_k \rangle_{L^2(D)} \cdot \langle \psi_{x,r}, \lambda_k^m f_k \rangle_{L^2(D)}\\
	&= (-1)^m r^{-2m} \sum_{k=1}^\infty \langle \psi_{x,r}, f_k \rangle_{L^2(D)} \cdot \langle (\Delta^m\psi)_{x,r}, f_k \rangle_{L^2(D)}\\
	&= (-1)^m r^{-2m} \langle \psi_{x,r} , (\Delta^m \psi)_{x,r} \rangle_{L^2(D)}
	=(-1)^m r^{-2d-2m} \int_D \psi(\tfrac{y-x}{r}) (\Delta^m \psi)(\tfrac{y-x}{r}) \, \d y\\
	& \le r^{-d-2m} \int_{\R^d} |\psi(z)| |\Delta^m \psi (z)| \, \d z = C_{m,\psi}\, r^{-d-2m}.
\end{align*}

Case 2: $\beta \not\in \N_0$. In this case, we can find $m \in \N_0$ such that $m < \beta < m+1$.
Take $p = 1/(m+1-\beta)>1$ and $q = 1/(\beta-m)>1$, so that $1/p+1/q=1$.
Set $\beta_1 = m/p$ and $\beta_2 = (m+1)/q$.
Note that $\beta_1+\beta_2 = \beta$.
Hence, we may apply H\"older's inequality and Case 1 to deduce that
\begin{align*}
	&\sum_{k=1}^\infty \lambda_k^{\beta} \left| \langle \psi_{x,r}, f_k \rangle_{L^2(D)} \right|^2\\
	&\le \left( \sum_{k=1}^\infty \lambda_k^m \left| \langle \psi_{x,r}, f_k \rangle_{L^2(D)} \right|^2\right)^{m+1-\beta} \left(\sum_{k=1}^\infty \lambda_k^{m+1} \left| \langle \psi_{x,r}, f_k \rangle_{L^2(D)} \right|^2\right)^{\beta-m}\\
	& \le \left(C_{m,\psi} \, r^{-d-2m} \right)^{m+1-\beta} \left( C_{m+1,\psi} \, r^{-d-2m-2} \right)^{\beta-m}
	= C_{\beta, \psi} r^{-d-2\beta}.
\end{align*}
This proves \eqref{SLND:claim:J}.
It follows that
\[
	J\lesssim r^{-d-2\alpha-6}.
\]
Now, we return to \eqref{SLND:I:CS} and use \eqref{SLND:I:value} to deduce that
\begin{align}\label{SLND:claim:proved}
    \E\Bigg[\bigg(u(t,x)-\sum_{j=1}^n a_j u(t_j,x_j)\bigg)^2\Bigg]
    \gtrsim
    r^{-2d-4} J^{-1}
    \gtrsim
    r^{2\alpha-d+2}
    =
    r^{2\theta},
\end{align}
uniformly for all $n \ge 1$, $(t,x),(t_1,x_1),\dots, (t_n,x_n) \in [0,T]\times D$ and $a_1,\dots, a_n \in \R$.
The factors $T$ and $R_D$ in \eqref{SLND:r} can be enlarged and factored out to give a lower bound.
Moreover, since
\[
    \rho((t,x),(t_j,x_j))
    \asymp
    \bigl(|t-t_j|^{1/2}\vee |x-x_j|\bigr)^{\theta},
\]
the implicit constant above can also be factored and absorbed by $C$ to obtain \eqref{SLND:claim}.
This completes the proof of the proposition.
\end{proof}


\subsection{Proof of Theorem \ref{thm:optimal}}

\begin{proof}
Under \eqref{Dalang}, the upper bound in \eqref{var:u-u} follows from Propositions \ref{prop:var:u}, \ref{prop:u-u:t}, \ref{prop:u-u:x}.
The lower bound in \eqref{var:u-u} follows from an application of \eqref{SLND:claim:proved} with $n=1$, $a_1=1$ and $(t_1,x_1)=(s,y)$, and another application of \eqref{SLND:claim:proved} with $(t,x)$ and $(s,y)$ swapped.
The SLND property on $[\delta,T]\times D'$ follows from Proposition \ref{prop:SLND} and the variance bound \eqref{var:u-u}, which implies that $\Var(u(t,x)-u(s,y)) \asymp \rho((t,x),(s,y))^2$ uniformly for all $(t,x),(s,y)\in [\delta,T]\times D'$ since $t,s$ and $d(x),d(y)$ are all bounded above and below.
\end{proof}

\section{Proofs of Theorems \ref{thm:moc}--\ref{thm:sbp}}\label{s:proofs}


\subsection{Harmonizable representation}

We use the framework of Lee and Xiao \cite{LX23} to prove Theorems \ref{thm:moc}--\ref{thm:sbp}. Their first assumption
requires an independently scattered Gaussian field indexed by Borel subsets
of \([0,\infty)\), together with quantitative estimates for its low- and
high-frequency parts. We construct this field from the temporal Fourier
transform of the heat kernel.

Let \(\{W_n^{(1)}\}_{n\ge1}\) and \(\{W_n^{(2)}\}_{n\ge1}\) be independent
sequences of real-valued centered Gaussian white noises on \(\R\), and set
\(W_n=W_n^{(1)}+iW_n^{(2)}\). Thus, for \(f,g\in L^2(\R;\mathbb C)\),
\[
    \E\left[\left(\Re\int_\R f\,\d W_n\right)
             \left(\Re\int_\R g\,\d W_m\right)\right]
    =\delta_{nm}\Re\int_\R f(\tau)\overline{g(\tau)}\,\d\tau,
\]
where $\Re$ denotes real part.
Define
\begin{align}\label{q}
    q_{n,t}(\tau)
    =\frac{\e^{-i\tau t}-\e^{-\lambda_n t}}{\lambda_n-i\tau}.
\end{align}
For any Borel set \(A\subset [0,\infty)\), define the centered Gaussian field
\begin{equation}\label{eq:harm}
    u(A,t,x)
    =\frac1{\sqrt{2\pi}}
    \sum_{n=1}^{\infty}\lambda_n^{-\alpha/2}f_n(x)
    \Re\int_{\{\tau \in \R:\,\lambda_n^{\theta/2}\vee|\tau|^{\theta/2}\in A\}}
    q_{n,t}(\tau)\,W_n(\d\tau).
\end{equation}
Since \(q_{n,t}\) is the Fourier transform of
\(r\mapsto\1_{(0,t)}(r)\e^{-\lambda_n(t-r)}\), Plancherel's theorem and Proposition \ref{prop:spec:rep} imply that, under condition \eqref{Dalang}, $\{u([0,\infty),t,x)\}_{t \ge 0, x \in D}$ has the same covariance, and hence the same law, as $\{ u(t,x) \}_{t \ge 0, x \in D}$.
Therefore, we may and will identify $u([0,\infty),\cdot) = u(\cdot)$ for the rest of the paper.
Following \cite{Balan12,DMX17,LX23}, we call \eqref{eq:harm} the harmonizable representation of $u$.

\begin{proposition}\label{prop:LX}
Suppose \(D\) is a bounded \(C^2\) domain and
\eqref{alpha:Holder:range} holds.
Let \(R=R_1 \times R_2\) be a
compact rectangle in $(0,\infty)\times D$. Write \(z=(z_0,\ldots,z_d)=(t,x_1,\ldots,x_d)\)
and set
\[
    \alpha_0=\theta/2,\qquad \alpha_i=\theta\ \, (1\le i\le d),
    \qquad
    \gamma_i=\alpha_i^{-1}-1\ \, (0\le i \le d).
\]
Then every \(\gamma_i\) is positive, and the Gaussian field \(u(A,z)\) from
\eqref{eq:harm} satisfies the following properties.

\begin{enumerate}
\item For every \(z\in R\), the map \(A\mapsto u(A,z)\) is an
independently scattered Gaussian noise,
\(u([0,\infty),z)=u(z)\), and the fields \(u(A,\cdot)\) and
\(u(B,\cdot)\) are independent for disjoint \(A,B\). 
Thus Assumption~2.1(a) of \cite{LX23} holds.\smallskip

\item There exists $C_R>0$ such that for every \(1\le a<b\le\infty\) and \(z,z'\in R\),
\begin{gather*}
\qquad\qquad \|u([a,b),z)-u(z)-u([a,b),z')+u(z')\|_2\le C_R\left[
   \sum_{i=0}^d a^{\gamma_i}|z_i-z'_i|+b^{-1}\right],\\
	\|u([0,1),z) - u([0,1),z')\|_2 \le C_R \sum_{i=0}^d |z_i-z_i'|,
\end{gather*}
with \(0^{\gamma_i}=0\) and \(\infty^{-1}=0\). So, Assumption~2.1(b) of \cite{LX23} holds with \(a_0=1\).\smallskip

\item Let \(z^0=(0,\dots, 0)\in\R^{d+1}\). There exists \(c_0>0\) such that
\[
    \Var\left(u(z)\mid u(z^1),\ldots,u(z^n)\right)
    \ge
    c_0\min_{0\le j\le n}\rho(z,z^j)^2
\]
for all \(n\ge1\) and \(z,z^1,\ldots,z^n\in R\). Thus
Assumption~2.2 of \cite{LX23} holds.\smallskip

\item There exist $c_1,c_2>0$ such that the metric $d_u(z,z') := \|u(z)-u(z')\|_2$ satisfies
\[
    c_1\rho(z,z')\le d_u(z,z')\le c_2\rho(z,z'),
    \qquad \forall z,z'\in R,
\]
so Assumption~2.3 of \cite{LX23} holds.
\end{enumerate}
%
\end{proposition}

\begin{proof}
First, under condition \eqref{alpha:Holder:range}, \(0<\theta<1\), so
\(\gamma_0=(\theta/2)^{-1}-1>0\) and
\(\gamma_i=\theta^{-1}-1>0\) for \(1\le i\le d\).
Item (1) follows from the definition \eqref{eq:harm} and the fact that $W_n$ are independent Gaussian noises. 

To prove item (2), we claim that there exists $C>0$ such that
\begin{align*}
	&\|u([a,b),t,x)-u(t,x)-u([a,b),s,y)+u(s,y)\|_{2}\\
	&\hskip1.5in \le C\bigl(a^{\gamma_0}|t-s|+a^{\gamma_1}|x-y|+b^{-1}\bigr),\\
	& \|u([0,1),t,x) - u([0,1),s,y)\|_2 \le C (|t-s|+|x-y|)
\end{align*}
for all $1\le a <b \le \infty$ and $(t,x),(s,y) \in R$.

The following identities will be used throughout the proof:
\[
    \frac{2}{\theta}(1+d/2-\alpha)=2\gamma_0,
    \qquad
    \frac{2}{\theta}(d/2-\alpha)=2\gamma_1,
    \qquad
    \frac{2}{\theta}(d/2-\alpha-1)=-2 .
\]
Since
\[
\begin{aligned}
&u([a,b),t,x)-u(t,x)-u([a,b),s,y)+u(s,y)  \\
&\quad =
    -\bigl(u([0,a),t,x)-u([0,a),s,y)\bigr)
    -\bigl(u([b,\infty),t,x)-u([b,\infty),s,y)\bigr),
\end{aligned}
\]
its $\|\cdot\|_2$ norm is bounded by
\[
    I_a+J_b(t,x)+J_b(s,y),
\]
where
\[
    I_a=\|u([0,a),t,x)-u([0,a),s,y)\|_2,
    \qquad
    J_b(t,x)=\|u([b,\infty),t,x)\|_2 .
\]

We first estimate $J_b(t,x)$. If \(b=\infty\), then \(J_b=0\).
Let \(1<b<\infty\) and set
\[
	\Lambda_b=b^{2/\theta}.
\] 
By Wiener isometry and
\(|\e^{-i\tau r}-\e^{-\lambda_n r}|\le2\),
\[
    |J_b(t,x)|^2
    \lesssim
    \sum_{n=1}^\infty
    \lambda_n^{-\alpha}|f_n(x)|^2
    \int_{\{\tau:\lambda_n^{\theta/2}\vee|\tau|^{\theta/2}\ge b\}}
    \frac{\d\tau}{\lambda_n^2+\tau^2}.
\]
If \(\lambda_n<\Lambda_b\), then the domain of integration is contained in
\(\{\tau: |\tau|\ge\Lambda_b\}\), and
\[
\begin{aligned}
    \sum_{n \ge 1:\lambda_n<\Lambda_b}
    \lambda_n^{-\alpha}|f_n(x)|^2
    \int_{|\tau|\ge\Lambda_b}
    \frac{\d\tau}{\lambda_n^2+\tau^2}
    &\lesssim
    \Lambda_b^{-1}
    \sum_{n \ge 1: \lambda_n<\Lambda_b}\lambda_n^{-\alpha}|f_n(x)|^2  \\
    &\lesssim
    \Lambda_b^{d/2-\alpha-1}
    =
    b^{-2},
\end{aligned}
\]
thanks to Lemma~\ref{lem:spectral:1}. If \(\lambda_n\ge\Lambda_b\), then the domain of integration is $\R$ and
\[
    \int_\R\frac{\d\tau}{\lambda_n^2+\tau^2}=\frac{\pi}{\lambda_n}.
\]
Thus by Lemma \ref{lem:spectral:2},
\begin{align*}
    \sum_{n \ge 1: \lambda_n\ge\Lambda_b}
    \lambda_n^{-\alpha}|f_n(x)|^2
    \int_\R\frac{\d\tau}{\lambda_n^2+\tau^2}
    &\lesssim
    \sum_{n \ge 1:\lambda_n\ge\Lambda_b}
    \lambda_n^{-(\alpha+1)}|f_n(x)|^2\\
    &\lesssim
    \Lambda_b^{d/2-\alpha-1}
    =
    b^{-2}.
\end{align*}
Consequently, \(J_b(t,x)\lesssim b^{-1}\) for all $1<b\le \infty$ and $(t,x) \in R$.

Next, we prove that $I_a \lesssim a^{\gamma_0}|t-s| + a^{\gamma_1} |x-y|$ for all $a \ge 1$ and $(t,x),(s,y) \in R$.
Note that it is enough to prove this for all $|t-s| \le \varepsilon_0$ and $|x-y| \le \varepsilon_0$, where $\varepsilon_0>0$ is a fixed number, because $\|u([0,a),t,x)\|_2 \le \|u(t,x)\|_2$ and Lemma  \ref{lem:var:u:1} implies that $\|u(t,x)\|_2$ is uniformly bounded.
Let $a \ge 1$ and set 
\[
	\Lambda_a=a^{2/\theta}.
\]
Recall \eqref{q}. 
By Wiener isometry,
\[
    I_a^2
    \lesssim
    \sum_{n \ge 1: \lambda_n<\Lambda_a}
    \lambda_n^{-\alpha}
    \int_{|\tau|<\Lambda_a}
    |f_n(x)q_{n,t}(\tau)-f_n(y)q_{n,s}(\tau)|^2\,\d\tau .
\]
Split the last expression as \(A_1+A_2\), where
\[
    A_1:=
    \sum_{n \ge 1: \lambda_n<\Lambda_a}
    \lambda_n^{-\alpha}|f_n(x)|^2
    \int_{|\tau|<\Lambda_a}|q_{n,t}(\tau)-q_{n,s}(\tau)|^2\,\d\tau,
\]
and
\[
    A_2:=
    \sum_{n \ge 1: \lambda_n<\Lambda_a}
    \lambda_n^{-\alpha}|f_n(x)-f_n(y)|^2
    \int_{|\tau|<\Lambda_a}|q_{n,s}(\tau)|^2\,\d\tau .
\]
Since $|\e^{-i\tau t}-\e^{-i\tau s}|\le|\tau||t-s|$ and $|\e^{-\lambda_n t}-\e^{-\lambda_n s}|\le\lambda_n|t-s|$,
we have 
\[
	|q_{n,t}(\tau)-q_{n,s}(\tau)|\lesssim |t-s|.
\]
Hence, by Lemma \ref{lem:spectral:1},
\[
    A_1
    \lesssim
    |t-s|^2\Lambda_a
    \sum_{n \ge 1: \lambda_n<\Lambda_a}\lambda_n^{-\alpha}|f_n(x)|^2
    \lesssim
    |t-s|^2\Lambda_a^{1+d/2-\alpha}
    =
    a^{2\gamma_0}|t-s|^2 .
\]
Also, we have
\[
	|q_{n,s}(\tau)|\le \frac2{(\lambda_n^2+\tau^2)^{1/2}}
	\quad \text{and}\quad 
	\int_\R |q_{n,s}(\tau)|^2\,\d\tau\lesssim\lambda_n^{-1}.
\]
Thus,
\[
    A_2
    \lesssim
    \sum_{n \ge 1: \lambda_n<\Lambda_a}
    \lambda_n^{-(\alpha+1)}|f_n(x)-f_n(y)|^2 .
\]
Let $\varepsilon_0 \in (0,1)$ and $C_0>0$ be the constants given by Lemma \ref{lem:curve}.
By that lemma, for any $x,y \in R_2$ with $|x-y|\le\varepsilon_0$, we can find a $C^2$ curve $\gamma:[0,1] \to D$ such that $\gamma(0)=x$, $\gamma(1) = y$ and 
\[
	|\gamma'(r)| \le C_0|x-y| \quad \text{for all $r \in [0,1]$.}
\]
Then, by the fundamental theorem of calculus and Cauchy-Schwarz inequality,
\begin{align*}
    |f_n(x)-f_n(y)|^2 
    &=
    \left|\int_0^1 \nabla f_n(\gamma(r))\cdot\gamma'(r)\,\d r\right|^2
    \le C_0^2 |x-y|^2 \int_0^1 |\nabla f_n(\gamma(r))|^2\,\d r.
\end{align*}
This together with Lemma~\ref{lem:spectral:3} yields
\begin{align*}
    A_2
    &\lesssim
    |x-y|^2 \int_0^1
    \sum_{n \ge 1: \lambda_n<\Lambda_a}
    \lambda_n^{-(\alpha+1)}|\nabla f_n(\gamma(r))|^2\,\d r\\
    &\lesssim
    |x-y|^2\Lambda_a^{d/2-\alpha}
    =
    a^{2\gamma_1}|x-y|^2 .
\end{align*}
Combining \(A_1\) and \(A_2\), we have
\[
    I_a
    \lesssim
    a^{\gamma_0}|t-s|+a^{\gamma_1}|x-y|.
\]
Since \(|x-y|\asymp \sum_{i=1}^d|x_i-y_i|\), the last display and the
previous estimates for $J_b$ give exactly the first bound in item (2). 
The estimate for $I_a$ above with $a=1$ gives the second bound.
This proves item (2).

Item (3) follows from Proposition \ref{prop:SLND} and
$\inf_{(t,x) \in R}[t^\theta \wedge (d(x))^{2\theta}] > 0$ because the compact rectangle $R \subset (0,\infty) \times D$ has a strictly positive distance from $\{0\} \times D$ and $[0,T] \times \partial D$.
Finally, item (4) follows from \eqref{var:u-u}.
\end{proof}

\subsection{Proof of Theorem~\ref{thm:moc}}

\begin{proof}
Fix $z_0 \in (0,\infty) \times D$.
Choose a compact rectangle \(R \subset (0,\infty) \times D\) with \(z_0\) in its interior. Proposition~\ref{prop:LX}
verifies Assumptions~2.1 and 2.3 of \cite{LX23}. Therefore, Theorem~5.2 of 
\cite{LX23} applies to $\{u(t,x)\}_{(t,x) \in R}$ and gives the asserted exact local
modulus with respect to the metric \(\rho\), with a constant $K_0$ that is finite and strictly positive.
\end{proof}

\subsection{Proof of Theorem~\ref{thm:umoc}}

We first recall the notion of metric entropy.
Let $d_u(z,z') = \|u(z)-u(z')\|_2$ denote the canonical metric for $u$ on $[0,\infty) \times D$.
For any set $A \subset [0,\infty)\times D$, let $N(A,r)$ be the entropy number, defined as the smallest number of $d_u$-balls of radius $r$ needed to cover $A$.

\begin{lemma}\label{lem:N}
Suppose $D\subset \R^d$ is a bounded $C^2$ domain and \eqref{alpha:Holder:range} holds.
Fix $T>0$. Then, there exists $C>0$ such that
\begin{align*}
	N([0,T]\times D, r) \le Cr^{-Q} \quad \forall r \in (0,1],
\end{align*}
where $Q$ is given by \eqref{eq:Q}.
\end{lemma}

\begin{proof}
For any $r>0$, define
\[
	D_{r} := \{ x \in D: d(x) > r^{1/\theta} \},
\]
where $\theta$ and $d(x)$ are given by \eqref{theta} and \eqref{d(x)}, respectively.
By Theorem \ref{thm:optimal}, there exist $c_0,c_1>0$ such that
\begin{align}\begin{split}\label{d_u:bd}
	&c_0 \left[ \rho((t,x),(s,y)) \wedge \left( (\sqrt{t} \wedge d(x))^\theta \vee (\sqrt{s} \wedge d(y))^\theta \right) \right]\\
	&\le d_u((t,x),(s,y)) \le c_1 \left[ \rho((t,x),(s,y)) \wedge  \left( (\sqrt{t} \wedge d(x))^\theta \vee (\sqrt{s} \wedge d(y))^\theta \right)  \right]
\end{split}\end{align}
for all $(t,x), (s,y) \in [0,T] \times D$.
For any $r \in (0,r_0)$, we write $[0,T]\times D = E_r \cup F_r$, where
\[
	E_r = [(r/c_1)^{2/\theta}, T] \times D_{r/c_1} \quad \text{and} \quad 
	F_r = ([0,T]\times D) \setminus E_r.
\]
By \eqref{d_u:bd}, $F_r$ can be covered by a single $d_u$-ball of radius $r$, thus \[
	N(F_r, r) = 1.
\]
To estimate $N(E_r,r)$, let $B_1,\dots, B_p$ be a maximal set of pairwise disjoint $d_u$-balls of radius $r/2$ with centers in $E_r$. 
Maximality implies that the balls $2B_1,\dots, 2B_p$ form a cover for $E_r$, where $2B_i$ is the $d_u$-ball with radius $r$ and the same center as $B_i$.
Hence, $N(E_r,r) \le p$.
Since $D$ is a bounded domain and $B_1,\dots, B_p$ are pairwise disjoint, their total volume $\sum_{i=1}^p \mathrm{Vol}(B_i)$ is uniformly bounded.
By \eqref{d_u:bd}, $N(E_r,r) r^Q \le p r^Q\lesssim \sum_{i=1}^p \mathrm{Vol}(B_i) = O(1)$, hence we can find $C_0>0$ such that
\[
	N(E_r,r) \le p \le C_0 r^{-Q}.
\]
Therefore, we have $N([0,T]\times D, r) \le 1 + C_0 r^{-Q} \lesssim r^{-Q}$ for all $r \in (0,1]$.
\end{proof}

\begin{proof}[Proof of Theorem~\ref{thm:umoc}]
First, by \eqref{var:u-u}, we have
\[
	\lim_{\varepsilon\to0^+} \sup_{\substack{z,z'\in [0,T]\times D\\ 0<\rho(z,z') \le \varepsilon}} \frac{d_u(z,z')}{\rho(z,z') \sqrt{\log(1/\rho(z,z'))}} = 0.
\]
Hence, we may use Lemma 7.1.1 of \cite{MR} to deduce that \eqref{umoc} holds for some constant $K_1 \in [0,\infty]$.
It remains to show that $0<K_1<\infty$.

To show that $K_1<\infty$, we apply Theorem 1.3.5 of \cite{AT} to see that there exist a universal constant $K>0$ and a random variable $\eta>0$ such that a.s., for all $\varepsilon \in (0,\eta]$,
\begin{align*}
	\sup_{\substack{z,z'\in [0,T]\times D\\\rho(z,z') \le \varepsilon}} |u(z)-u(z')| \le K \int_0^\varepsilon \sqrt{\log N([0,T]\times D, r)} \, \d r.
\end{align*}
By Lemma \ref{lem:N}, $N([0,T]\times D, r) \le C r^{-Q}$ for all $r \in (0,1]$.
Hence, we may let $\varepsilon_n = e^{-n}$ and proceed as in the proof of Theorem 3.13 of \cite{HL26} to deduce that there exist constants $C_1, C_2 \in (0,\infty)$ such that a.s., for all $n$ large,
\[
	\sup_{\substack{z,z'\in [0,T]\times D\\\rho(z,z') \le \varepsilon_n}} |u(z)-u(z')| 
	\le C_1 \varepsilon_n \sqrt{\log(1/\varepsilon_n)}
\]
and
\[
	\lim_{n \to \infty} \sup_{\substack{z,z'\in [0,T]\times D\\\rho(z,z') \le \varepsilon_n}} \frac{|u(z)-u(z')|}{\rho(z,z')\sqrt{\log(1/\rho(z,z'))}} \le C_2 \quad \text{a.s.}
\]
This implies that $K_1 \le C_2 <\infty$.

To show that $K_1>0$, we choose and fix any compact rectangle $R \subset (0,T] \times D$.
By Proposition \ref{prop:LX}, $\{u(t,x)\}_{(t,x) \in R}$ satisfies Assumptions~2.1--2.3 of \cite{LX23}.
Hence, Theorem~6.1 of \cite{LX23} implies that
\[
    \lim_{\varepsilon\to0^+}
    \sup_{\substack{z,z'\in R\\0<\rho(z,z')\le\varepsilon}}
    \frac{|u(z)-u(z')|}
    {\rho(z,z')\sqrt{\log(1/\rho(z,z'))}}
    =
    C
    \qquad\text{a.s.}
\]
for some constant $C \in (0,\infty)$.
Clearly, we have $K_1 \ge C>0$.
This completes the proof of Theorem \ref{thm:umoc}.
\end{proof}

\subsection{Proof of Theorem~\ref{thm:lil}}

\begin{proof}
Fix \(z_0\in(0,\infty)\times D\) and choose a compact rectangle
\(R \subset (0,\infty) \times D\) with \(z_0\) in its interior. 
Under \eqref{alpha:Holder:range}, Proposition~\ref{prop:LX} shows that $\{u(t,x)\}_{(t,x) \in R}$ satisfies Assumptions~2.1 and 2.2 of 
\cite{LX23} with \(\alpha_0=\theta/2\) and
\(\alpha_i=\theta\) for \(1\le i\le d\).
Therefore, we may apply Theorem~4.4 of \cite{LX23} to obtain the asserted Chung's law of the iterated logarithm with a  constant in \((0,\infty)\). 
\end{proof}

\subsection{Proof of Theorem~\ref{thm:sbp}}


%

\begin{proof}[Proof of Theorem \ref{thm:sbp}]
By Lemma 2.2 of Talagrand \cite{Talagrand} (see also \cite[Lemma 3.4]{DLMX} for a more precise statement) and Lemma \ref{lem:N}, there exists $K>0$ such that 
\[
	\P\left\{ \sup_{z,z' \in [0,T]\times D}|u(z)-u(z')| \le \varepsilon \right\} \ge e^{-K/\varepsilon^Q}\qquad
	\forall \varepsilon \in (0,1].
\]
Fix $t_0 \in (0,T]$ and $x_0 \in D$.
Thanks to the Gaussian correlation inequality and the preceding, we can find $K_3>0$ such that
\begin{align*}
	\P\left\{ \sup_{z \in [0,T]\times D}|u(z)| \le \varepsilon \right\}
	&\ge \P\left\{ |u(t_0,x_0)| \le \frac{\varepsilon}{2} \right\} \, \P\left\{ \sup_{z,z' \in [0,T]\times D}|u(z)-u(z')| \le \frac{\varepsilon}{2} \right\} \\
	& \ge \P\left\{ |Z| \le \frac{\varepsilon}{2\|u(t_0,x_0)\|_2} \right\} e^{-K2^Q/\varepsilon^Q}
	\ge e^{-K_3/\varepsilon^Q}
\end{align*}
for all $\varepsilon \in (0,1]$, where $Z$ has a standard normal distribution, and we have used $\Var(u(t,x)) \asymp t^{\theta} \wedge (d(x))^{2\theta}$, which follows from  \eqref{var:u-u}, and $\P\{|Z| \le a \} \ge C_A a$ for $|a| \le A$ to obtain the last inequality. This proves the lower bound in Theorem \ref{thm:sbp}.

To show the upper bound, we fix a compact rectangle $R \subset (0,T] \times D$.
By Theorem \ref{thm:optimal}, $\|u(z)-u(z')\|_2 \asymp \rho(z,z')$ uniformly for all $z,z' \in R$ and $\{u(z)\}_{z \in R}$ satisfies the SLND property.
Hence, we may apply Theorem 5.1 of \cite{X09} to find $K_4>0$ such that
\[
	\P\left\{ \sup_{z \in R} |u(z)| \le \varepsilon \right\} \le e^{-K_4 \varepsilon^{-Q}}
\]
uniformly for all $\varepsilon \in (0,1]$.
This implies the upper bound in Theorem \ref{thm:sbp}.
\end{proof}

{\bf Acknowledgments.}
C.Y. Lee was supported in part by the Shenzhen Peacock grant
2025TC0013.

\bibliographystyle{plain}
\bibliography{SHE}

\end{document}